\documentclass[journal,twoside,web]{ieeecolor}
\usepackage{generic}
\usepackage{cite}
\usepackage{amsmath,amssymb,amsfonts}
\usepackage{algorithmic}
\usepackage{graphicx}
\usepackage{textcomp}
\usepackage{mathrsfs} 								
\usepackage{tikz,pgfplots,datatool}
\usepackage{color,colortbl}		
\usepackage{arydshln}
\definecolor{xgray}{rgb}{0.75, 0.75, 0.75}
\newcommand{\hdl}{\hdashline[2pt/1pt]}
\newcommand{\hgl}{\noalign{\global\arrayrulewidth=0.6pt}\arrayrulecolor{xgray}\hline\noalign{\global\arrayrulewidth=0.4pt}\arrayrulecolor{black}}
\pdfmapfile{=stix.map}
\DeclareFontFamily{U}{stixextrai}{}
\DeclareFontShape{U}{stixextrai}{m}{n}
{ <-> stix-mathtt }{}

\let\originalleft\left
\let\originalright\right
\renewcommand{\left}{\mathopen{}\mathclose\bgroup\originalleft}
\renewcommand{\right}{\aftergroup\egroup\originalright}
\newcommand{\N}{\mathbb{N}}
\newcommand{\R}{\mathbb{R}}
\newcommand{\Ahb}{\mathbb{A}}
\newcommand{\Bhb}{\mathbb{B}}
\newcommand{\Chb}{\mathbb{C}}
\newcommand{\Dhb}{\mathbb{D}}
\newcommand{\Xhb}{\mathbb{X}}
\newcommand{\Ac}{\mathcal A}
\newcommand{\Bc}{\mathcal B}
\newcommand{\Cc}{\mathcal C}
\newcommand{\Dc}{\mathcal D}
\newcommand{\Kc}{\mathcal K}
\newcommand{\Lc}{\mathcal L}
\newcommand{\Pc}{\mathcal P}
\newcommand{\Xc}{\mathcal X}
\newcommand{\Yc}{\mathcal Y}
\newcommand{\Zc}{\mathcal Z}
\newcommand{\Gc}{\mathcal G}

\newcommand{\Ah}{\hat{A}}
\newcommand{\Bh}{\hat{B}}
\newcommand{\Ch}{\hat{C}}
\newcommand{\Dh}{\hat{D}}
\newcommand{\Eh}{\hat{E}}
\newcommand{\Fh}{\hat{F}}
\newcommand{\eps}{\varepsilon}
\newcommand{\al}{\alpha}
\newcommand{\ga}{\gamma}
\newcommand{\De}{\Delta}
\newcommand{\te}[1]{\text{\ \ #1\ \ }}
\newcommand{\sy}{(\bullet)^T}
\newcommand{\eig}{\operatorname{eig}}
\newcommand{\diag}{\operatorname{diag}}
\newcommand{\tr}{\operatorname{tr}}
\newcommand{\col}{\operatorname{col}}
\newcommand{\mat}[2]{\left(\begin{array}{#1}#2\end{array}\right)}
\newcommand{\matl}[4]{\left#1\begin{array}{#2}#3\end{array}\right#4}
\newcommand{\smat}[1]{\left(\begin{smallmatrix}#1\end{smallmatrix}\right)}
\newcommand{\opt}{\text{opt}}
\newcommand{\cl}{\prec}
\newcommand{\cg}{\succ}
\renewcommand{\t}{\tilde}
\newcommand{\ti}{\times}
\newcommand{\Hz}{\mathcal{H}_2}
\newcommand{\Hi}{\mathcal{H}_\infty}
\newcommand{\RP}{\mathcal{R}_p}
\newtheorem{thm}{Theorem}
\newtheorem{lemm}{Lemma}
\newtheorem{rema}{Remark}
\newtheorem{defn}{Definition}
\newtheorem{prob}{Problem}
\newtheorem{exam}{Example}
\newcommand{\Sb}{\mathbb{S}}
\newcommand{\wh}{\hat{w}}
\newcommand{\zh}{\hat{z}}
\newcommand{\He}{\text{He}}
\newcommand{\bu}{\bullet}
\newcommand{\bs}[1]{\boldsymbol{#1}}
\newcommand{\Tr}{\R_{\scalebox{0.6}{$(\bullet\,0)$}}}
\newcommand{\Trf}{\R_{\scalebox{0.6}{$(\bullet\,\bullet)$}}}
\newcommand{\Trp}{\R_{\scalebox{0.7}{$\left(\hspace{-0.2ex}\begin{smallmatrix}\bullet\,0\\[-0.2ex]\bullet\,0\end{smallmatrix}\hspace{-0.2ex}\right)$}}}
\newcommand{\Tc}{\R_{\scalebox{0.7}{$\left(\hspace{-0.2ex}\begin{smallmatrix}0\\[-0.2ex]\bullet\end{smallmatrix}\hspace{-0.2ex}\right)$}}}
\newcommand{\Tcf}{\R_{\scalebox{0.7}{$\left(\hspace{-0.2ex}\begin{smallmatrix}\bullet\\[-0.2ex]\bullet\end{smallmatrix}\hspace{-0.2ex}\right)$}}}
\newcommand{\Tcp}{\R_{\scalebox{0.7}{$\left(\hspace{-0.2ex}\begin{smallmatrix}0\,0\\[-0.2ex]\bullet\,\bullet\end{smallmatrix}\hspace{-0.2ex}\right)$}}}
\newcommand{\T}{\R_{\scalebox{0.7}{$\left(\hspace{-0.2ex}\begin{smallmatrix}\bullet\,0\\[-0.2ex]\bullet\,\bullet\end{smallmatrix}\hspace{-0.2ex}\right)$}}}
\newcommand{\Ur}{{\mathscr{R}_{\scalebox{0.6}{$(0\,\bullet)$}}}}
\newcommand{\U}{{\mathscr{U}_{\scalebox{0.7}{$\left(\hspace{-0.2ex}\begin{smallmatrix}0\,\bullet\\[-0.2ex]0\,0\end{smallmatrix}\hspace{-0.2ex}\right)$}}}}
\newcommand{\Lo}{\mathscr{L}_{\scalebox{0.7}{$\left(\hspace{-0.2ex}\begin{smallmatrix}\bullet\,0\\[-0.2ex]\bullet\,\bullet\end{smallmatrix}\hspace{-0.2ex}\right)$}}}

\def\BibTeX{{\rm B\kern-.05em{\sc i\kern-.025em b}\kern-.08em
    T\kern-.1667em\lower.7ex\hbox{E}\kern-.125emX}}
\begin{document}
\title{Gain-Scheduling Controller Synthesis for Nested Systems with Full Block Scalings}
\author{Christian A. R\"osinger and Carsten W. Scherer, \IEEEmembership{Fellow, IEEE}
\thanks{
Funded by Deutsche Forschungsgemeinschaft (DFG, German Research Foundation) under Germany's Excellence Strategy - EXC 2075 - 390740016. We acknowledge the support by the Stuttgart Center for Simulation Science (SimTech).
}
\thanks{The authors are with the Institute of Mathematical Methods in the Engineering Sciences, Numerical Analysis and Geometrical Modeling, Department of Mathematics, University of Stuttgart, 70569 Stuttgart, Germany (e-mail: christian.roesinger@imng.uni-stuttgart.de, carsten.scherer@imng.uni-stuttgart.de).
\textit{(Corresponding author: Christian~A.~R\"osinger.)}
}
}

\maketitle

\begin{abstract}
This work presents a framework to synthesize structured gain-scheduled controllers for structured plants whose dynamics change according to time-varying scheduling parameters. Both the system and the controller are assumed to admit descriptions in terms of a linear time-invariant system in feedback with so-called scheduling blocks, which collect all scheduling parameters into a static system. We show that such linear fractional representations permit to exploit a so-called lifting technique in order to handle several structured gain-scheduling design problems. These could arise from a nested inner and outer loop control configuration with partial or full dependence on the scheduling variables. Our design conditions are formulated in terms of convex linear matrix inequalities and permit to handle multiple performance objectives.
\end{abstract}


\section{Introduction}
\IEEEPARstart{I}{n} this work, we consider gain-scheduled synthesis based on linear fractional representations (LFRs) for
the standard configuration in Fig.~\ref{figs}, as motivated by the early works \cite{packard1994}, \cite{apkarian1995}.
Here, $G(\De)$ is a linear parametrically-varying (LPV) system affected by some matrix-valued time-varying uncertainty $\De$ whose current value can change arbitrarily fast and is measured online. For instance, in a concrete application, $\De$ can represent the rotor speed of the generators of a wind turbine \cite{tien2016}, or the variation of the longitudinal speed in a car \cite{mustaki2019}.
The philosophy of gain-scheduling synthesis \cite{packard1994, becker1995, apkarian1995, wu1996, helmersson98, scorletti1998} is based on the idea to design a $\De$-dependent controller $K(\De)$ in Fig.~\ref{figs} which achieves better performance
if compared to a robust controller that does not depend on $\Delta$.

We present a flexible synthesis framework encompassing gain-scheduled problems for different nested interconnections of LPV systems. As inspired by \cite{voulgaris2000},
one specific configuration covered by our approach is shown in Fig.~\ref{fig0}, to which we refer as \textit{partial gain-scheduling} in the sequel. This configuration involves an outer loop with a linear time invariant (LTI) plant $P_2$ and a controller $C_2$, interconnected  with a gain-scheduled inner loop consisting of an LPV system $P_1(\De)$ and a scheduled controller $C_1(\De)$.
\begin{figure}
	\begin{center}
		\includegraphics[trim=2 0 2 2, clip,height=0.126\textwidth]{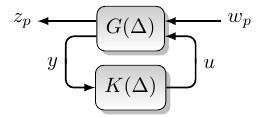}
		\caption{Gain-scheduling configuration}
		\label{figs}
	\end{center}
\end{figure}
Note that the outer loop with $P_2$ and $C_2$ is affected by $P_1(\De)$ and $C_1(\De)$ by one-sided communication links $\xi$ and $\eta$, respectively. Such nested configurations are of practical interest, e.g., in the control of induction motors \cite{blanchini2012}, where the physical constraints impose a fast $\De$-dependent inner loop with $\De$ being the rotor speed of the motor, and a slow outer mechanical loop.
Further, our framework permits us to handle the case that $P_2=P_2(\De)$ and $C_2=C_2(\De)$ are also $\De$-dependent in Fig.~\ref{fig0}, which is called \textit{triangular gain-scheduling} for reasons to be seen later. Such structures emerge, for instance, in a wind park if wind turbines are interacting in a nested fashion and where $\De$ depends on the wind speed and some torque coefficients~\cite{bucc2019}.

Since we are interested in embedding these problems into a unifying framework for analysis and synthesis, we show how to translate these nested configurations into Fig.~\ref{figs} by making use of the flexibility of LFRs. This leads to controller design problems for particularly structured $G(\De)$ and $K(\De)$. One of the main contributions of this work is to show, for the first time, that these structured design problems can be solved by convex optimization techniques. This is achieved through a general design framework for nested gain-scheduled control problems based on linear matrix inequalities (LMIs).

Moreover, a central aspect of our design framework lies in the flexibility for handling a combination of different criteria in one shot, such as, e.g., stability, $\Hi$-and $\Hz$-performance objectives.
\begin{figure}
	\begin{center}
		\includegraphics[trim=6 0 6 2.5, clip, height=0.247\textwidth]{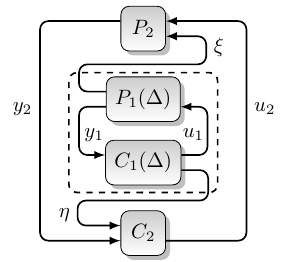}
		\caption{Nested gain-scheduling loop}
		\label{fig0}
	\end{center}
\end{figure}
Since $\Hz$-control requires to guarantee finiteness of the closed-loop norm, we also show how to incorporate the recent approaches \cite{roesinger2019b}, \cite{roesinger2020} based on $D$-, positive real and full block scalings into our framework. These works focus on a certain structured $\Hz$-design problem to render the direct feedthrough term of $w_p\to z_p$ zero in Fig.~\ref{figs}, which in turn guarantees finiteness of the closed-loop $\Hz$-norm by design.

On the one hand, if $P_1(\De)$ and $C_1(\De)$ are $\De$-independent LTI systems in Fig.~\ref{fig0}, LMI solutions are given for nominal, nested $\Hi$- and $\Hz$-design in \cite{scherer2013}, \cite{scherer2014}, while \cite{lessard2015} uses coupled Riccati equations to solve the $\Hz$-case. On the other hand, without the outer loop in Fig.~\ref{fig0}, a gain-scheduled $\Hi$-solution with full block scalings is given, e.g., in \cite{scherer2000} to design a suitable $\De$-dependent controller $C_1(\De)$ for some parameter-dependent plant $P_1(\De)$. In order to handle the $\Hz$-analogue, we have recently shown in \cite{roesinger2020} how to use the so-called \textit{lifting technique} in the context of gain-scheduling. This lifting technique embeds the original gain-scheduled problem into some new design framework such that synthesis can be performed using LMIs. As our main technical contribution, we show that lifting is the key enabling technique to also handle nested gain-scheduling. This includes Fig.~\ref{fig0} (partial gain-scheduling) and triangular gain-scheduling, which turns out to be the most challenging case since it involves coupled $\De$-structures between the inner and the outer loop.
To the best knowledge of the authors, no other methods exist to solve the partial/triangular gain-scheduling problem in this generality.

The paper is organized as follows. After introducing some notation, we illustrate the main design steps for a special $\Hz$-gain-scheduling problem without nested structures in Sec.~\ref{sech2}. The exposition is tailored to the seamless extension to nested gain-scheduling in Sec.~\ref{secvi}, with the corresponding analysis and synthesis conditions presented for multiple objectives in Secs.~\ref{sec3} and \ref{secsyn}, respectively. We conclude this work by giving an illustrative numerical example in Sec.~\ref{secnum}.

\textbf{Notation.}
We give the basic notations here, and particular ones for structured matrices and inequalities in Secs.~\ref{secvi} and~\ref{sec3}, respectively.
If $\N_0$ is the set of nonnegative integers,
$\N_0^p$ denotes
the set of $p$-tuples  $a=(a_1,\ldots,a_p)\in\N_0^p$ of length $|a|:=a_1+\cdots+a_p$. If $m,n\in\N_0^p$,
we denote by $\R^{m\ti n}$
the set of real $|m|\ti|n|$ matrices that carry a row/column-partition as induced by the entries of the tuples $m$/$n$, while $\Sb^n\subset\R^{n\ti n}$ is the associated subset of partitioned real symmetric matrices.
Further, $I$ is the identity matrix, an $n$-partition of which being specified as $I_n:=\diag(I_{n_1},\ldots,I_{n_p})$.
We use $\bullet$ for irrelevant matrix entries and $\col(M_1,\ldots,M_k):=(M_1^T,\ldots,M_k^T)^T$ for vectors or matrices $M_1,\ldots,M_k$.
If $P\in\R^{m\ti m}$, $M\in\R^{m\ti n}$, let $\tr(P)$ be the trace of $P$, $\eig(P)$ be the set of eigenvalues of $P$, and we write $\sy PM:=M^TPM$ and $\He(P):=P+P^T$.

\section{A special case: $\Hz$-gain-scheduling}\label{sech2}

After giving a brief introduction of the gain-scheduled synthesis problem, we
motivate the cornerstones of our design procedure, including the lifting technique, for the situation of $\Hz$-gain-scheduling without nested constraints.

Let us consider the standard loop for gain-scheduling in Fig.~\ref{figs}. For some given value set $\mathbf{V}=\text{Co}\{\De_1,\ldots,\De_N\}$, the convex hull of finitely many matrices $\De_i\in\R^{r_\De\ti s_\De}$, let $0\in\mathbf{V}$ and assume that $\De$ lies in the set $\mathbf{\De}:=C([0,\infty),\mathbf{V})$ of matrix-valued, arbitrarily fast time-varying (continuous) uncertainties.
Moreover, $G(\De)$ is assumed to admit the LFR
\begin{equation}\label{e1}
\mat{c}{\dot{x}\\\hline \hat{z}\\ z_p\\ y}=\matl({c|ccc}{
	\Ah_{11}&\Ah_{12}&\Bh_1^p&\Bh_1\\\hline
	\Ah_{21}&\Ah_{22}&\Bh_2^p&\Bh_2\\
	\Ch_1^p&\Ch_2^p&\Dh^p&\Eh\\
	\Ch_1&\Ch_2& \Fh&\Dh})\mat{c}{x\\\hline\hat{w}\\ w_p\\ u},\quad
\hat{w}=\hat{\De}(\De)\hat{z}
\end{equation}
where $\hat{\De}:\mathbf{V}\to\R^{r\times s}$ is an affine map and $\Delta$ is contained in $\mathbf{\De}$; note that the dimension of $\hat{\De}(\De)$ might differ from that of $\De$.
In this representation, $\hat{w}\to \hat{z}$ is the uncertainty channel, $w_p\to z_p$ serves to impose
performance specifications while $u\to y$ is the channel to interconnect LPV controllers.
Recall that this encompasses standard LFRs with $\hat{\De}(\De)$ admitting a block-diagonal
structure with several repeated time-varying parametric uncertainties on the diagonal
\cite{packard1994}, \cite{apkarian1995}, \cite{zhou1996}; see \cite{zhou1996} for a general introduction to LFRs.

\begin{exam}\label{ex1}
	If $G(\De)$ in Fig.~\ref{figs} is given as the uncertain system $\smat{\dot{x}\\y}=\smat{\De^2&3\\4&0}\smat{x\\u}$
with $\Delta\in \mathbf{\De}$ where $\mathbf{V}=[-1,1]$, we sequentially define $\hat{w}_i:=\De\hat{z}_i$ for $i=1,2$ with $\hat{z}_1:=x$, $\hat{z}_2:=\hat{w}_1$ to obtain the LFR
	$$
	\mat{c}{\dot{x}\\\hline \hat{z}_1\\\hat{z}_2\\ y}=\mat{c|ccc}{0&0&1&3\\\hline 1&0&0&0\\0&1&0&0\\4&0&0&0}\mat{c}{x\\\hline \hat{w}_1\\\hat{w}_2\\ u},\ \mat{c}{\hat{w}_1\\\hat{w}_2}=\mat{cc}{\De&0\\0&\De}\mat{c}{\hat{z}_1\\\hat{z}_2};
	$$
this indeed matches (\ref{e1}) with block-diagonal $\hat{\De}(\De):=\smat{\De&0\\0&\De}$.
\end{exam}

Analogously to the plant, we describe the gain-scheduled controller $K(\De)$ in Fig.~\ref{figs} as an LFR
\begin{equation}\label{e2}
\mat{c}{\dot{x}_c\\\hline \zh_{c}\\ u}=
\matl({c|cc}{\Ah_{11}^c&\Ah_{12}^c&\Bh^c_1\\\hline \Ah_{21}^c&\Ah_{22}^c&\Bh^c_2\\  \Ch^c_1&\Ch^c_2&\Dh^c})
\mat{c}{x_c\\\hline \wh_c\\ y},
\quad \wh_{c}=\hat{\De}_c(\De)\zh_{c}
\end{equation}
with $\De\in\mathbf{\De}$ and a so-called scheduling function
$$
\hat{\De}_c:\mathbf{V}\to
\R^{r^c\ti s^c},\ \De\mapsto\hat{\De}_c(\De).
$$

The interconnection of \eqref{e1} and \eqref{e2} (through the shared signals $u$ and $y$) then admits the LFR
\begin{equation}\label{eq3}
\mat{c}{\dot{x}_{e}\\ \zh_e\\ z_p}=\mat{ccc}{\hat{\Ac}_{11}&\hat{\Ac}_{12}&\hat{\Bc}_{1}\\ \hat{\Ac}_{21}&\hat{\Ac}_{22}&\hat{\Bc}_{2}\\\hat{\Cc}_{1}&\hat{\Cc}_{2}&\hat{\Dc}}\mat{c}{x_{e}\\ \wh_e\\ w_p},\quad \wh_e=\hat{\De}_{e}(\De)\zh_e
\end{equation}
with the closed-loop signals $x_{e}:=\col(x,x_c)$, $\zh_e:=\col(\zh,\zh_c)$, $\wh_e:=\col(\wh,\wh_c)$ and the extended scheduling block
\begin{equation}\label{dele}
\hat{\De}_{e}(\De):=
\diag(\hat{\De}(\De),\hat{\De}_c(\De))
\end{equation}
of dimension $(r,r^c)\times (s,s^c)$. Recall that the closed-loop matrices $\hat{\Ac}_{ij},\,\hat{\Bc}_i,\,\hat{\Cc}_j,\,\hat{\Dc}$ can be obtained by standard computations, as shown in Sec.~\ref{seclifting} for a different scenario.

Gain-scheduling synthesis then means to
design matrices $\Ah^c_{ij},\,\Bh^c_i,\,\Ch^c_j,\,\Dh^c$ and a possibly nonlinear
scheduling function $\hat{\De}_c(.)$ such that the controlled system \eqref{eq3} satisfies a desired performance specification for all $\De\in\mathbf{\De}$, as made precise in the sequel.
We emphasize that the controller \eqref{e2} is indeed gain-scheduled, in the sense that it
requires knowledge of the value of $\hat{\De}_c(\De(t))$ at each time instant $t\geq 0$ for its implementation; explicit bounds on the to-be-designed controller order and the size of $\hat{\De}_c(.)$ are given for each (nested) synthesis result. For logical similarities, we denote the system matrices associated to the integrator and the scheduling parameter in a similar fashion by using different indices, as e.g., $\Ah_{ij},\Bh_i,\,\Ch_j$ in \eqref{e1}.

\subsection{Problem formulation}\label{classes4}
Let us now formulate the $\Hz$-gain scheduling problem.
We assume that the direct feedthrough term of $u\to y$ in (\ref{e1}) vanishes, i.e. $\Dh=0$. This assumption ensures well-posedness of the controlled interconnection \eqref{eq3}, and, as essential for synthesis in Sec.~\ref{syn4}, renders the closed-loop matrices affinely dependent on the controller matrices (\ref{e2}).	
As widely spread over the existing gain-scheduling literature, we
consider the problem without any other structural constraints in the LFRs (\ref{e1}), (\ref{e2}), i.e., the scheduling function and the describing matrices are unstructured with
$$
\begin{aligned}
&\mat{cccc}{
	\Ah_{11}&\Ah_{12}&\Bh_1^p&\Bh_1\\
	\Ah_{21}&\Ah_{22}&\Bh_2^p&\Bh_2\\
	\Ch_1^p&\Ch_2^p&\Dh^p&\Eh\\
	\Ch_1&\Ch_2& \Fh&\Dh}
\in\mat{llll}{\R^{n\ti n}&\R^{n\ti r}&\R^{n\ti m^p}&\R^{n\ti m}\\\R^{s\ti n}&\R^{s\ti r}&\R^{s\ti m^p}& \R^{s\ti m}\\\R^{k^p\ti n}&\R^{k^p\ti r}&\R^{k^p\ti m^p}&\R^{k^p\ti m}\\\R^{k\ti n}&\R^{k\ti r}&\R^{k\ti m^p}&0_{k\ti m}},\\
&\quad \hat{\De}(\De)=\De\ \  \text{being of dimension $r\ti s=r_\De\times s_\De$,}
\end{aligned}
$$
and
$$
\begin{aligned}
&\mat{ccc}{\Ah^c_{11}&\Ah^c_{12}&\Bh^c_1\\\Ah^c_{21}&\Ah^c_{22}&\Bh^c_2\\\Ch^c_1&\Ch^c_2&\Dh^c}\in\mat{ccc}{\R^{n^c\ti n^c}&\R^{n^c\ti r^c}&\R^{n^c\ti k}\\\R^{s^c\ti n^c}&\R^{s^c\ti r^c}& \R^{s^c\ti k}\\\R^{m\ti n^c}&\R^{m\ti r^c}&\R^{m\ti k}}.
\end{aligned}
$$
This is addressed as the \textit{unstructured gain-scheduling problem} and briefly expressed by
$$
G(\De)\in\Gc_1\quad \text{and}\quad K(\De)\in\Kc_1.
$$
Recall that the controlled LFR (\ref{eq3}) is called \textit{well-posed} if $I-\hat{\De}_{e}(\De)\hat{\Ac}_{22}$ is non-singular for all $\De\in\mathbf{V}$.
For brevity, we call (\ref{eq3}) \textit{stable} if the system is exponentially stable, i.e., there exist constants $\al>0$ and $c\geq0$ such that every solution of (\ref{eq3}) for $w_p=0$ and for any $x_e(0)\in\R^n$, $\De\in\mathbf{\De}$ satisfies $\|x_{e}(t)\|\leq ce^{-\al t}\|x_{e}(0)\|$ for all $t\geq 0$. Since our scaling parameter $\De\in\mathbf{\De}$ is time-varying in \eqref{eq3}, we use the definition of the $\Hz$-norm for linear time-varying systems in the stochastic context \cite{pagfer00}.
\begin{prob}\label{problem1}
	For a given plant $G(\De)\in\Gc_1$ and $\ga>0$, find a controller $K(\De)\in\Kc_1$ such that the closed-loop  (\ref{eq3}) is well-posed, stable, and such that the squared $\Hz$-norm of $w_p\to z_p$ is smaller than $\ga$ for $x_e(0)=0$ and for all $\De\in\mathbf{\De}$.
\end{prob}

\subsection{Closed-loop analysis for original systems}\label{h2sec22}

For $\Hz$-performance, we suppose that the plant and controller LFR (\ref{e1}), (\ref{e2}) are built such that, after interconnecting $\hat{\De}_e(\De)$ in (\ref{eq3}), the direct feedthrough term of $ w_p\to z_p$ vanishes. In Sec.~\ref{classes3}, we show that our general design procedure comes with the strong advantage that we can enforce this condition with tailored LFRs for (\ref{e1}), (\ref{e2}). Under this hypothesis, let us recall a well-known analysis result based on the full block $S$-procedure. By using the class of multipliers
\begin{equation}\label{posconstr}
\mathbf{\hat{P}}:=\left\{\hat{\Pc}\in\Sb^{(r,r^c,s,s^c)}\ \big|\
\sy \hat{\Pc}\smat{\hat{\De}_{e}(\De)\\I_{(s,s^c)}}\cg0\
\forall
\De\in\mathbf{V}\right\}
\end{equation}
for the extended scheduling block \eqref{dele}, to which we also refer as \textit{full block scalings}, we can characterize the requirements in Problem~\ref{problem1} by the feasibility of two
standard matrix inequalities as follows \cite{scherer2000}.
\begin{thm}\label{theo:analysish2}
	Problem~\ref{problem1} is solved for $G(\De)\in\Gc_1$, $K(\De)\in\Kc_1$ if there exist $\Xc_1\cg0$ and $\hat{\Pc}\in\mathbf{\hat{P}}$, $Z\cg0$ with $\tr(Z)<1$ such that the closed-loop system (\ref{eq3}) fulfills
	\begin{equation}\label{eqan1}
	\begin{aligned}
	\sy\mat{cc;{2pt/1pt}c;{2pt/1pt}c}{0&\Xc_1&0&0\\\Xc_1&0&0&0\\\hdl 0&0&\hat{\Pc}&0\\\hdl 0&0&0&-\gamma I}\mat{ccc}{I_{(n,n^c)}&0&0\\\hat{\Ac}_{11}&\hat{\Ac}_{12}&\hat{\Bc}_{1}\\\hdl 0&I_{(r,r^c)}&0\\\hat{\Ac}_{21}&\hat{\Ac}_{22}&\hat{\Bc}_{2}\\\hdl 0&0&I}&\cl0,\\
	\sy\mat{c;{2pt/1pt}c;{2pt/1pt}c}{-\Xc_1&0&0\\\hdl 0&\hat{\Pc}&0\\\hdl 0&0&Z^{-1}}\mat{cc}{I_{(n,n^c)}&0\\\hdl 0&I_{(r,r^c)}\\\hat{\Ac}_{21}&\hat{\Ac}_{22}\\\hdl \hat{\Cc}_{1}&\hat{\Cc}_{2}}&\cl0.
	\end{aligned}
	\end{equation}
\end{thm}
\phantom{space}

Theorem~\ref{theo:analysish2} forms the basis for a convex characterization of the existence of a gain-scheduled controller \eqref{e2} such that the formulated analysis conditions are satisfied for some full block scaling $\hat{\Pc}\in\mathbf{\hat{P}}$.
Technically, all existing scaling approaches to obtain such gain-scheduled synthesis results
are based on the elimination of the ingredients defining the controller, as seen
for special plants in \cite{scherer2000}.
However, it is well-known that, even for nominal synthesis, such an elimination step is
infeasible for the full $\Hz$-conditions in (\ref{eqan1}), since two inequalities are involved
which are coupled through the matrices $\hat{\Ac}_{21},\,\hat{\Ac}_{22}$.
It is among the key contributions of this paper to overcome this deficiency through what we call {\em lifting for gain-scheduling}.

\subsection{Closed-loop analysis for lifted systems}\label{seclifting}

The motivation for lifting is as follows. If $\hat{\Pc}\in\mathbf{\hat{P}}$ is fixed, the anti-diagonal structure of $\smat{0&\Xc_1\\\Xc_1&0}$ in \eqref{eqan1} is essential to obtain convex conditions for synthesizing nominal controllers with a suitable transformation of $\Ah^c_{ij}$, $\Bh^c_i$, $\Ch^c_j$ and $\Dh^c$ in \eqref{e2}, as addressed in detail in
\cite{masubuchi1998}, \cite{scherer1997}. In gain-scheduling synthesis, $\hat{\Pc}$ is an unstructured variable and it remains fully unclear how to modify the transformation from \cite{masubuchi1998}, \cite{scherer1997} to convexify \eqref{eqan1}.
It is a decisive innovation of this paper to overcome this trouble for (even nested) gain-scheduling synthesis by lifting the descriptions of the system and the controller such that, in the resulting analysis
conditions, $\hat{\Pc}$ is replaced by an anti-diagonally structured block
$\smat{0&\Pc\\\Pc&0}$ with a suitable new scaling matrix $\Pc$.
In our notation, we use a hat for the initial system components \eqref{e1}-\eqref{dele} and the scalings \eqref{posconstr} to distinguish them from the lifted descriptions denoted without a hat in the sequel.

Lifting amounts to building an augmented LFR of $G(\De)$ in (\ref{e1}) as follows. The uncertainty equation
$\hat{w}=\hat{\De}(\De)\hat{z}$ in (\ref{e1}) can be expressed as $\hat{w}=-\hat{w}+2\hat{\De}(\De)\hat{z}$ which leads to $w=\De_l(\De)z$ for $\De\,{\in}\,\mathbf{\De}$ with
\begin{equation}\label{liftedblock}
w\,{:=}\,z\,{:=}\,\mat{c}{\hat{w}\\\hat{z}},\ \ \De_l(\De)\,{:=}\,\mat{cc}{-I_r&2\hat{\De}(\De)\\0&I_s}.
\end{equation}
By employing analogous steps for the LTI system in \eqref{e1}, we arrive at an equivalent reformulation of the overall system as
\begin{equation}\label{eq5}
\begin{aligned}
&\mat{c}{\dot{x}\\\hline z\\\hdl z_p\\\hdl y}=\mat{c|c;{2pt/1pt}c;{2pt/1pt}c}{A_{11}&A_{12}&B_1^{p}&B_1\\\hline A_{21}&A_{22}&B_2^{p}&B_2\\\hdl C_1^{p}&C_2^{p}&D^{p}&E\\\hdl C_1&C_2&F&0}
\mat{c}{x\\\hline w\\\hdl w_p\\\hdl u}\\
&=
	\mat{c|cc;{2pt/1pt}c;{2pt/1pt}c}{
	\Ah_{11}&\Ah_{12}&0&\Bh_1^p&\Bh_1\\\hline
	0&I_r&0&0&0\\
	2\Ah_{21}&2\Ah_{22}&-I_s&2\Bh_2^p&2\Bh_2\\\hdl
	\Ch_1^p&\Ch_2^p&0&\Dh^p&\Eh\\\hdl
	\Ch_1&\Ch_2& 0&\Fh&0}
\mat{c}{x\\\hline w\\\hdl w_p\\\hdl u},\ \ w=\De_l(\De)z.
\end{aligned}
\end{equation}
Note that the augmented uncertainty channel $w\to z$ is square and of size $|l|\ti |l|$ for the partition $l:=(r,s)$.

\begin{defn}\label{liftedLFR}
	We abbreviate (\ref{eq5}) by $G_l(\De)$ and call it \textit{lifted LFR}, while we refer to $\De_l(\De)$ from (\ref{liftedblock}) as the \textit{lifted block}. Further, we say that $G_l(\De)\in\Gc^l_1$ if $G(\De)\in\Gc_1$.
\end{defn}

By defining  $w_c\,{:=}\,z_c\,{:=}\,\col(\hat{w}_c,\hat{z}_c)$, the controller (\ref{e2}) is lifted accordingly to
\begin{equation}\label{e2lift}
\begin{aligned}
\mat{c}{\dot{x}_c\\\hline z_{c}\\\hdl u}&=
\mat{c|cc;{2pt/1pt}c}{\Ah_{11}^c&\Ah_{12}^c&0&\Bh_1^c\\\hline 0&I_{r^c}&0&0\\ 2\Ah_{21}^c&2\Ah_{22}^c&-I_{s^c}&2\Bh_2^c\\\hdl \Ch_1^c&\Ch_2^c& 0&\hat{D}^c}
\mat{c}{x_c\\\hline w_c\\\hdl y},\\
w_{c}&=\De_c(\De)z_c:=\mat{cc}{-I_{r^c}&2\hat{\De}_c(\De)\\0&I_{s^c}}z_c
\end{aligned}
\end{equation}
with a square lifted scheduling block $\De_c(.)$ of dimension $|l^c|\ti|l^c|$ for the partition $l^c:=(r^c,s^c)$.
Interconnecting \eqref{eq5} and \eqref{e2lift} gives rise to the closed-loop LFR
\begin{equation}\label{eq3lifted}
\mat{c}{\dot{x}_{e}\\ z_e\\ z_p}=\mat{ccc}{{\Ac}_{11}&{\Ac}_{12}&{\Bc}_{1}\\ {\Ac}_{21}&{\Ac}_{22}&{\Bc}_{2}\\{\Cc}_{1}&{\Cc}_{2}&{\Dc}}\mat{c}{x_{e}\\ w_e\\ w_p}
\end{equation}
with $z_e:=\col(z,z_c)$, $w_e:=\col(w,w_c)$ and
\begin{equation}\label{dell}
w_e=\De_{e}(\De)z_e:=\diag(\De_l(\De),\De_c(\De))z_e.
\end{equation}
This leads to a key link between the initial and lifted setting.
\begin{lemm}[Lifting Lemma]\label{theo:untolift}
There exists a permutation matrix $\Pi$ of size $|l|+|l^c|$ such that
$\Xc_1$, $Z$, $\hat{\Pc}\in\mathbf{\hat{P}}$ fulfill
\eqref{eqan1},
\begin{equation}\label{ek1}
	\sy\hat{\Pc}\mat{c}{I_{(r,r^c)}\\0}\cl0,\qquad\sy\hat{\Pc}\mat{c}{0\\I_{(s,s^c)}}\cg0
\end{equation}
for the initial closed-loop interconnection \eqref{e1}-\eqref{eq3} if and only if
$\Xc_1$, $Z$, $\Pc:=\frac{1}{2}\Pi^T\hat{\Pc}\Pi$ satisfy
	\begin{equation}\label{eqan2}
	\begin{aligned}
	\sy\mat{cc;{2pt/1pt}cc;{2pt/1pt}c}{0&\Xc_1&0&0&0\\\Xc_1&0&0&0&0\\\hdl 0&0&0&\Pc&0\\0&0&\Pc&0&0\\\hdl 0&0&0&0&-\ga I}\mat{ccc}{I_{(n,n^c)}&0&0\\\Ac_{11}&\Ac_{12}&\Bc_{1}\\\hdl 0&I_{(l,l^c)}&0\\\Ac_{21}&\Ac_{22}&\Bc_{2}\\\hdl 0&0&I}&\cl0,\\
	\sy\mat{c;{2pt/1pt}cc;{2pt/1pt}c}{-\Xc_1&0&0&0\\\hdl 0&0&\Pc&0\\0&\Pc&0&0\\\hdl 0&0&0&Z^{-1}}\mat{cc}{I_{(n,n^c)}&0\\\hdl 0&I_{(l,l^c)}\\\Ac_{21}&\Ac_{22}\\\hdl \Cc_{1}&\Cc_{2}}&\cl0
	\end{aligned}
	\end{equation}
and
\begin{equation}\label{ek2}
	\sy \mat{cc}{0&\Pc\\\Pc&0}\mat{cc}{\De_{e}(\De)\\I}\cg0\quad
	\text{for all}\quad \De\in\mathbf{V}
\end{equation}
for the lifted closed-loop interconnection \eqref{eq5}-\eqref{eq3lifted}.
\end{lemm}

\begin{proof}
	Let the first inequality in \eqref{ek1} be true for $\hat{\Pc}\in\mathbf{\hat{P}}$.
If $\De\in\mathbf{V}$, we infer by the definition of $\mathbf{\hat{P}}$ the inequality
	\begin{multline*}
		\text{He}\left[\mat{cc}{-I_{(r,r^c)}&\hat{\De}_{e}(\De)\\0&I_{(s,s^c)}}^T(\tfrac{1}{2}\hat{\Pc})\mat{cc}{I_{(r,r^c)}&\hat{\De}_{e}(\De)\\0&I_{(s,s^c)}}\right]=\\
		=\mat{cc}{-\sy\hat{\Pc}\smat{I_{(r,r^c)}\\0}&0\\0&\sy\hat{\Pc}\smat{\hat{\De}_{e}(\De)\\I_{(s,s^c)}}}\cg0.
	\end{multline*}
By a congruence transformation with
$\smat{-I_{(r,r^c)}&\hat{\De}_{e}(\De)\\0&I_{(s,s^c)}}$ and if inserting $\hat{\De}_e(\De)$ from \eqref{dele}, the latter inequality is equivalent to
$$
0\cl\text{He}\left[(\tfrac{1}{2}\hat{\Pc})\mat{cc}{-I_{(r,r^c)}&\smat{2\hat{\De}(\De)&0\\0&2\hat{\De}_c(\De)}\\0&I_{(s,s^c)}}\right].
$$
If recalling \eqref{dell}, this can be equivalently expressed as
\begin{equation}\label{permatrix}
 0\,{\cl}\,\text{He}\left[\!(\tfrac{1}{2}\hat{\Pc})\Pi\De_{e}(\De)\Pi^T\right]\ \,\text{with}\ \,
\Pi\,{:=}\mat{cccc}{I_r&0&0&0\\0&0&I_{r^c}&0\\0&I_s&0&0\\0&0&0&I_{s^c}}
\end{equation}
being a permutation matrix. By a congruence transformation with $\Pi$, the inequality in \eqref{permatrix} transforms into \eqref{ek2} for the multiplier $\Pc=\frac{1}{2}\Pi^T\hat{\Pc}\Pi$.
The converse is shown by reversing the steps.
If $\Pc=\frac{1}{2}\Pi^T\hat{\Pc}\Pi$ holds, an analogous computation shows that (\ref{eqan1}) and the second inequality in (\ref{ek1}) are equivalent to (\ref{eqan2}), which is omitted for reasons of space.
\end{proof}

The lifting lemma can be interpreted as follows. For the subclass of scalings $\hat{\Pc}\in\mathbf{\hat{P}}$ with \eqref{ek1}, we can equivalently reformulate
the analysis conditions for the original interconnection \eqref{e1}-\eqref{eq3} to the ones based on
\eqref{eqan2} for the lifted interconnection \eqref{eq5}-\eqref{eq3lifted} and multipliers satisfying \eqref{ek2}. At the level of the lifted systems,
however, we are confronted with a highly structured controller (\ref{e2lift}), and it is unknown how  to formulate
convex conditions for synthesizing such controllers. As a further central step, we drop the structural constraint and synthesize, instead, an unstructured LPV controller - still with a square scheduling block - for the lifted system. This is the key toward convexification as exposed in detail next.

To this end, let us assume that the LPV controller is again described as in \eqref{e2}, but now with $r^c=s^c=:l^c$ as motivated above.
The interconnection of (\ref{e2}) with \eqref{eq5} leads, again, to \eqref{eq3lifted}
with $z_e:=\col(z,\zh_c)$, $w_e:=\col(w,\wh_c)$ and scheduled as
\begin{equation}\label{delc}
w_e=\De_{lc}(\De)z_e:=\diag(\De_l(\De),\hat{\De}_c(\De))z_e
\end{equation}
of dimension $(r,s,l^c)\times(r,s,l^c)$.
Let us now compactly express the closed-loop matrices as
\begin{multline}\label{clformula}
\mat{c;{2pt/1pt}c}{\Ac_{ij}&\Bc_{i}\\\hdl \Cc_{j}&\Dc}=\\=
\mat{cc;{2pt/1pt}c}{A_{ij}&0&B^{p}_i\\0&0&0\\\hdl C^{p}_j&0&D^{p}}+\mat{cc}{0&B_i\\I&0\\\hdl0&E}\mat{cc}{\Ah^c_{ij}&\Bh^c_i\\\Ch^c_j&\Dh^c}\mat{cc;{2pt/1pt}c}{0&I&0\\C_j&0&F}.
\end{multline}
As motivated by Lemma~\ref{theo:untolift}, we introduce the scaling class
\begin{equation}\label{lscalings}
\mathbf{P}:=\left\{\Pc\in\Sb^{(l,l^c)}\ \big|\
\sy \smat{0&\Pc\\\Pc&0}\smat{\De_{lc}(\De)\\I}\cg0\
\forall\De\in\mathbf{V} \right\}
\end{equation}
to obtain the following result.
\begin{thm}\label{theo:analysislh2}
Let there exists a controller $K(\De)\in\Kc_1$ with $l^c=r^c=s^c$, $\Xc_1\cg0$, $Z\cg0$ with $\tr(Z)<1$, and $\Pc\in\mathbf{P}$ such that the closed-loop system (\ref{eq3lifted}), \eqref{delc} obtained for the lifted system $G_l(\De)\in\Gc_1^l$ satisfies (\ref{eqan2}).
Then, for this choice of $K(\De)$, $\Xc_1$ and $Z$, there exists a suitable multiplier $\hat{\Pc}\in\mathbf{\hat{P}}$ such that the closed-loop system (\ref{eq3}) obtained for the original system $G(\De)\in\Gc_1$ fulfills the analysis inequalities (\ref{eqan1}).
\end{thm}
\begin{proof}
Fix $\De\in\mathbf{V}$. Then $\Pc\in\mathbf{P}$ implies $\He\left[\Pc\De_{lc}(\De)\right]=\sy \smat{0&\Pc\\\Pc&0}\smat{\De_{lc}(\De)\\I}\cg0$.
If recalling the structures of $\De_l(.)$, $\De_{lc}(.)$ in \eqref{liftedblock} and \eqref{delc},
left-multiplying $\smat{\hat{\De}(\De)^T&I_{s}&0\\0&0&I_{l^c}}$ and right-multiplying the transpose leads to
	\begin{equation}\label{eqpr1}
	\He\left[\mat{cc}{\hat{\De}(\De)&0\\I_{s}&0\\\hdl 0&I_{l^c}}^T\Pc
	\mat{cc}{\hat{\De}(\De)&0\\I_{s}&0\\\hdl 0&\hat{\De}_c(\De)}
	\right]\cg0.
	\end{equation}
	In accordance with (\ref{eqpr1}), let us partition $\Pc$ as
	\begin{equation}\label{eqpr2}
	\Pc=\mat{c;{2pt/1pt}c}{Q&R\\\hdl S&T}=\mat{cc;{2pt/1pt}c}{Q_{11}&Q_{12}&R_1\\Q_{21}&Q_{22}&R_2\\\hdl S_1&S_2&T}
	\end{equation}
	and the left/right factor in (\ref{eqpr1}) as $\renewcommand{\arraystretch}{1.0}\arraycolsep=1.0pt\scalebox{0.8}{$\mat{c}{A\\\hdl C}^T$}\big/\renewcommand{\arraystretch}{1.0}\arraycolsep=1.0pt\scalebox{0.8}{$\mat{c}{A\\\hdl B}$}$ to infer
	$$
	0\,{\cl}\,\He\left[\mat{c}{A\\\hdl C}^T\!\!\mat{c;{2pt/1pt}c}{Q&R\\\hdl S&T}\mat{c}{A\\\hdl B}\right]\,{=}\,\sy\He\mat{ccc}{Q&R&0\\0&0&0\\S&T&0}\!\mat{c}{A\\B\\C}.
	$$
	After a suitable permutation and using (\ref{eqpr2}), this reads as
	$$
	\sy\underbrace{\He\mat{cc|cc}{Q_{11}&R_1&Q_{12}&0\\0&0&0&0\\\hline Q_{21}&R_2&Q_{22}&0\\S_1&T&S_2&0}}_{=:\hat{\Pc}}\mat{cc}{\hat{\De}(\De)&0\\0&\hat{\De}_c(\De)\\\hline I_{s}&0\\0&I_{l^c}}\cg0.
	$$
Since $\Delta\in\mathbf{V}$ was arbitrary, we infer $\hat{\Pc}\in\mathbf{\hat{P}}$. Analogous arguments applied to
(\ref{eqan2}) lead to (\ref{eqan1}) for the very same $\hat{\Pc}$.
\end{proof}
As a consequence, designing a general unstructured controller for the lifted plant based on the analysis inequalities
\eqref{eqan2} and multipliers $\Pc\in\mathbf{P}$ does lead to a solution of Problem~\ref{problem1}
for the original plant. As the key benefit, the lifted inequalities \eqref{eqan2} do indeed depend on $\Pc$ in the desired structural fashion.

\subsection{$\Hz$-Synthesis: Unstructured plant and controller}\label{syn4}

Let us be given $G(\De)\in\Gc_1$ as in  (\ref{e1}). Following \cite{scherer2000}, the information about the uncertainty
$\hat{\De}(\De)=\De$ is captured by the class $\mathbf{P}_p$  of \textit{primal scalings} as
\begin{equation}\label{primal}
\begin{aligned}
\mathbf{P}_p:=\left\{\right.&P\in\Sb^{(r,s)}\ \big|\
\sy P\smat{I_{r}\\0}\cl0\ \text{and}\\
&\left.\hspace{8ex} \sy P\smat{\hat{\De}(\De)\\I_{s}}\cg0\
\text{for all}\ \De\in\mathbf{V} \right\}.
\end{aligned}
\end{equation}
As motivated by standard results on integral quadratic constraints based on graph separation \cite{gohsaf1995}, \cite{iwasaka1998}, we also make use of the class of so-called \textit{dual scalings}
\begin{equation}\label{dual}
\begin{aligned}
\mathbf{P}_d
:=\left\{\right.&\t{P}\in\Sb^{(r,s)}\ \big|\
\sy \t{P}\smat{0\\I_{s}}\cg0\ \text{and}\\
&\left.\hspace{5.4ex} \sy \t{P}\smat{I_{r}\\-\hat{\De}(\De)^T}\cl0\
\text{for all}\ \De\in\mathbf{V} \right\}.
\end{aligned}
\end{equation}
All this permits us to present the synthesis conditions in order to solve Problem~\ref{problem1}.
We start by introducing the relevant design variables.
These comprise the symmetric unknowns
\begin{equation}\label{v1}
X_1\in\Sb^{n},\qquad Y_1\in\Sb^{n}
\end{equation}
in relation to the Lyapunov matrix $\Xc_1$ in (\ref{eqan2}), while the multiplier $\Pc$ is related to some elements
\begin{equation}\label{v2}
X_2\in \mathbf{P}_p,\qquad Y_2\in\mathbf{P}_d
\end{equation}
of the primal/dual scaling class (\ref{primal})/(\ref{dual}). Further, we take
\begin{equation}\label{v3}
\mat{ccc}{K_{11}&K_{12}&L_1\\K_{21}&K_{22}&L_2\\M_1&M_2&N}\in\mat{lll}{\R^{n\ti n}&\R^{n\ti l}&\R^{n\ti k}\\\R^{l\ti n}&\R^{l\ti l}& \R^{l\ti k}\\\R^{m\ti n}&\R^{m\ti l}&\R^{m\ti k}}.
\end{equation}

\begin{thm}\label{theo:synthesis1}
	There exists a controller $K(\De)\in\Kc_1$ with $l^c=r^c=s^c$ and $\Xc_1\cg0$, $Z\cg0$ with $\tr(Z)<1$, $\Pc\in\mathbf{P}$ such that the closed-loop system (\ref{eq3lifted}), \eqref{delc} obtained for the lifted system $G_l(\De)\in\Gc_1^l$ satisfies (\ref{eqan2}) iff there exist
(\ref{v1})-(\ref{v3}),  $Z\cg0$ with $\tr(Z)<1$  fulfilling the synthesis inequalities $\Xhb\cg 0$ and
	\begin{equation}\label{eqan3}
	\begin{aligned}
	\sy\mat{cc;{2pt/1pt}cc;{2pt/1pt}c}{0&I&0&0&0\\I&0&0&0&0\\\hdl 0&0&0&I&0\\0&0&I&0&0\\\hdl 0&0&0&0&-\ga I}\mat{ccc}{I_{(n,n)}&0&0\\\Ahb_{11}&\Ahb_{12}&\Bhb_{1}\\\hdl 0&I_{(l,l)}&0\\\Ahb_{21}&\Ahb_{22}&\Bhb_{2}\\\hdl 0&0&I}&\cl0,\\
	\sy\mat{c;{2pt/1pt}cc;{2pt/1pt}c}{-\Xhb&0&0&0\\\hdl 0&0&I&0\\0&I&0&0\\\hdl 0&0&0&Z^{-1}}\mat{cc}{I_{(n,n)}&0\\\hdl 0&I_{(l,l)}\\\Ahb_{21}&\Ahb_{22}\\\hdl \Chb_{1}&\Chb_{2}}&\cl0.
	\end{aligned}
	\end{equation}
	Here $\Ahb_{ij},\,\Bhb_i,\,\Chb_j,\,\Xhb$ for $i,j=1,2$ depend affinely on the decision variables and are defined for the lifted plant (\ref{eq5}) as
	\begin{equation}\label{s22}
	\begin{aligned}
	&\mat{c;{2pt/1pt}c}{\Ahb_{ij}&\Bhb_i\\\hdl \Chb_j&\Dhb}
	:=\mat{cc;{2pt/1pt}c}{A_{ij}Y_j&A_{ij}&B_i^{p}\\0&X_i^TA_{ij}&X_i^TB_i^{p}\\\hdl C_j^{p}Y_j&C_j^{p}&D^{p}}+\\&\hspace{13.4ex}+
	\mat{cc}{0&B_i\\I&0\\\hdl 0&E}
	\mat{cc}{K_{ij}&L_i\\M_j&N}
	\mat{cc;{2pt/1pt}c}{I&0&0\\0&C_j&F},
	\end{aligned}
	\end{equation}
	\begin{equation}\label{s22c}
	\hspace{-16.5ex}\Xhb:=\mat{cc}{Y_1&I\\I&X_1};
	\end{equation}
	the block $\Dhb$ is given for reasons of space and used in Sec.~\ref{secsyn}.\\
	If (\ref{eqan3}) holds, one can construct $\Xc_1\cg0$, $\Pc\in\mathbf{P}$, and $K(\De)\in\Kc_1$ with an affine scheduling function $\hat{\De}_c(.)$ and with $|n^c|\leq |n|$, $|l^c|\leq |l|$ such that (\ref{eqan2}) is satisfied.
\end{thm}

The synthesis conditions (\ref{eqan3}) reflect the structure of the lifted analysis inequalities (\ref{eqan2}) and reduce to standard LMIs in all variables and $\gamma$ after applying the Schur complement.
Since $\mathbf{V}=\text{Co}\{\De_1,\ldots,\De_N\}$ is a compact polytope and $\hat\Delta(.)$ is affine, the multiplier classes $\mathbf{P}_p$ and $\mathbf{P}_d$ in \eqref{primal} and \eqref{dual} do admit LMI representations \cite{scherer2000}, i.e., the inequalities in \eqref{primal}, \eqref{dual} hold for all $\De\in\mathbf{V}$ iff they are true for all vertices $\De\in\{\De_1,\ldots,\De_N\}$.
Therefore, all the conditions in Theorem~\ref{theo:synthesis1} translate into a
standard finite-dimensional LMI optimization problem, which allows to directly minimize the bound $\gamma$.

The following constructive proof of Theorem~\ref{theo:synthesis1} displays the convexification mechanism by exploiting the anti-diagonal structures with $\Xc_1$ and $\Pc$ in the lifted analysis condition (\ref{eqan2}).

\begin{proof}
	\textit{Necessity.} Suppose that (\ref{eqan2}) holds for some $\Xc_1\cg0$ and some scaling $\Xc_2:=\Pc\in\mathbf{P}$. Due to lifting and the resulting anti-diagonal structure, we can just push $\Xc_1$ and $\Xc_2$ in (\ref{eqan2})
 to the outer factors as
	\begin{equation}\label{pr1}
	\begin{aligned}
	\sy\mat{cc;{2pt/1pt}cc;{2pt/1pt}c}{0&I&0&0&0\\I&0&0&0&0\\\hdl 0&0&0&I&0\\0&0&I&0&0\\\hdl 0&0&0&0&-\ga I}\mat{ccc}{I_{(n,n^c)}&0&0\\\Xc_1\Ac_{11}&\Xc_1\Ac_{12}&\Xc_1\Bc_{1}\\\hdl 0&I_{(l,l^c)}&0\\\Xc_2\Ac_{21}&\Xc_2\Ac_{22}&\Xc_2\Bc_{2}\\\hdl 0&0&I}&\cl0,\\
	\sy\mat{c;{2pt/1pt}cc;{2pt/1pt}c}{-\Xc_1&0&0&0\\\hdl 0&0&I&0\\0&I&0&0\\\hdl 0&0&0&Z^{-1}}\mat{cc}{I_{(n,n^c)}&0\\\hdl 0&I_{(l,l^c)}\\\Xc_2\Ac_{21}&\Xc_2\Ac_{22}\\\hdl \Cc_{1}&\Cc_{2}}&\cl0.
	\end{aligned}
	\end{equation}
Then the synthesis inequalities (\ref{eqan3}) are obtained as follows.

	\textbf{Step 1} (\textit{Factorization of Lyapunov matrix $\Xc_1$}).\\
	W.l.o.g. let $n^c\geq |n|$. As shown in \cite{scherer1997}, this assumption allows to factorize $\Xc_1$ of dimension $|n|+|n^c|$ according to
	\begin{equation}\label{fac}
	\Xc_i\Yc_i=\Zc_i\ \ \text{with}\ \ \Yc_i:=\mat{cc}{Y_i&I\\\hdl V_i&0},\ \Zc_i:=\mat{cc}{I&X_i\\\hdl 0&U_i}
	\end{equation}
	for $i=1$ with the following properties: $\Yc_1$ has full column rank, and $X_1$, $Y_1$ satisfy \eqref{v1}.
	We infer $0\cl \Yc_1^T\Xc_1\Yc_1= \Zc_1^T\Yc_1$, which reads as $\Xhb\cg0$ for (\ref{s22c}) by symmetry.

	\textbf{Step 2} (\textit{Factorization of scaling matrix $\Xc_2$}).\\
	We target to factorize $\Xc_2$ of dimension $|l|+|l^c|$ as in (\ref{fac}) for $i=2$. For this purpose, we assume w.l.o.g. that $l^c\geq |l|$ and introduce the partition
	$
	\Xc_2=\smat{X_2&U_2^T\\U_2&Z_2}\in\Sb^{(l,l^c)}.
	$
	By perturbation, we can successively achieve invertibility of $Z_2$, a full column rank of $H_2:=-Z_2^{-1}U_2$ and invertibility of $X_2-U_2^TZ_2^{-1}U_2$, while (\ref{pr1}) persists to hold. Detailed arguments are given in the proof of the more general Theorem~\ref{theo:synthesis4}.
Then $\Xc_2$ is invertible and
	$
	\smat{Y_2\\V_2}:=\Xc_2^{-1}\smat{I_{l}\\0}
	$
	leads to (\ref{fac}) for $i=2$. By the block-inversion formula we infer $V_2=H_2Y_2$, which implies that $\Yc_2$ has full column rank.

	\textbf{Step 3} (\textit{Elimination of scheduling function $\hat{\De}_c(.)$}).\\
	Let us now show that (\ref{v2}) follows from eliminating the scheduling function $\hat{\De}_c(.)$ in $\De_{lc}(.)$.
 For reasons of space, we omit the arguments of the functions $\De_l(.)$, $\hat{\De}_c(.)$, $\De_{lc}(.)$.
 Since $\Xc_2\in\mathbf{P}$, an additive decomposition of $\De_{lc}$ leads to
	\begin{equation}\label{X2Dlc}
	\begin{aligned}
		0&\cl\sy \mat{cc}{0&\Xc_2\\\Xc_2&0}\mat{c}{\De_{lc}\\I}=\He\left[\Xc_2\De_{lc}\right]=\\
		&\hspace{10ex}=\He\left[\Xc_2\mat{cc}{\De_l&0\\0&0}\right]+\He\left[\Xc_2\mat{cc}{0&0\\0&\hat{\De}_c}\right].
	\end{aligned}
	\end{equation}
	By performing a congruence transformation with $\Yc_2$ and applying the factorization (\ref{fac}) for $i=2$, we get
	\begin{equation}\label{exdec}
		0\cl\He\left[\mat{c}{I\\\hdl X_2^T}\De_l\mat{c;{2pt/1pt}c}{Y_2&I}\right]+\He\left[\mat{c}{0\\\hdl U_2^T}\hat{\De}_c\mat{c;{2pt/1pt}c}{V_2&0}\right].
	\end{equation}
The diagonal blocks read as
\begin{equation}\label{diagb}
\He\left[\De_l Y_2\right]\cg0\te{and}\He\left[X_2^T\De_l\right]\cg0.
\end{equation}
Congruence transformations with $\smat{I_{r}&0\\-\hat{\De}^T&I_{s}}$ and $\smat{I_{r}&\hat{\De}\\0&I_{s}}$, respectively, imply $Y_{2}\in \mathbf{P}_d$ and $X_{2}\in\mathbf{P}_p$ and we get (\ref{v2}).

	\textbf{Step 4} (\textit{Convexifying parameter transformation}).\\
	Since $\Yc_i$ has full column rank for $i=1,2$, we can apply a congruence transformation with $\Yc_i/\Yc_i^T$ on the $i$-th block-column/-row of (\ref{pr1}) to infer with (\ref{fac}) that
	\begin{equation}\label{pr2}
	\begin{aligned}
		\sy\mat{cc;{2pt/1pt}cc;{2pt/1pt}c}{0&I&0&0&0\\I&0&0&0&0\\\hdl 0&0&0&I&0\\0&0&I&0&0\\\hdl 0&0&0&0&-\ga I}\mat{ccc}{I_{(n,n)}&0&0\\\Zc_1^T\Ac_{11}\Yc_1&\Zc_1^T\Ac_{12}\Yc_2&\Zc_1^T\Bc_{1}\\\hdl 0&I_{(l,l)}&0\\\Zc_2^T\Ac_{21}\Yc_1&\Zc_2^T\Ac_{22}\Yc_2&\Zc_2^T\Bc_{2}\\\hdl 0&0&I}&\cl0,\\
		\sy\mat{c;{2pt/1pt}cc;{2pt/1pt}c}{-\Xhb&0&0&0\\\hdl 0&0&I&0\\0&I&0&0\\\hdl 0&0&0&Z^{-1}}\mat{cc}{I_{(n,n)}&0\\\hdl 0&I_{(l,l)}\\\Zc_2^T\Ac_{21}\Yc_1&\Zc_2^T\Ac_{22}\Yc_2\\\hdl \Cc_{1}\Yc_1&\Cc_{2}\Yc_2}&\cl0.
	\end{aligned}
	\end{equation}
	By using the closed-loop formula (\ref{clformula}), we infer
	\begin{equation}\label{s19}
	\begin{aligned}
	\mat{c;{2pt/1pt}c}{\Zc^T_i\Ac_{ij}\Yc_j&\Zc_i^T\Bc_i\\\hdl \Cc_j\Yc_j&\Dc}
	&=\mat{cc;{2pt/1pt}c}{A_{ij}Y_j&A_{ij}&B_i^p\\0&X_i^TA_{ij}&X_i^TB_i^p\\\hdl C_j^pY_j&C_j^p&D^p}+\\&\hspace{-2ex}+
	\mat{cc}{0&B_i\\I&0\\\hdl 0&E}
	\mat{cc}{K_{ij}&L_i\\M_j&N}
	\mat{cc;{2pt/1pt}c}{I&0&0\\0&C_j&F}
	\end{aligned}
	\end{equation}
	for $i,j=1,2$ and with the substitution
	\begin{equation}\label{s2}
		\begin{aligned}
		&\hspace{-2ex}\mat{c;{2pt/1pt}c}{K_{ij}&L_i\\\hdl M_j&N}:=
		\mat{c;{2pt/1pt}c}{X_i^TA_{ij}Y_j&0\\\hdl 0&0}+\\
		&\hspace{4ex}+
		\mat{cc}{I&X_i^TB_i\\\hdl 0&I}
		\mat{cc}{U_i^T\hat{A}_{ij}^cV_j&U_i^T\hat{B}^c_i\\ \hat{C}^c_jV_j&\hat{D}^c}
		\mat{c;{2pt/1pt}c}{I&0\\ C_jY_j&I}.
		\end{aligned}
	\end{equation}
	Hence, (\ref{eqan3}) is true which finishes the proof of necessity.
	
	\textit{Sufficiency.} Suppose that (\ref{eqan3}) is satisfied with suitable decision variables. Then $V_1:=I-X_1^TY_1$ is invertible due to $\Xhb\cg0$, and, by perturbation of $X_2,\,Y_2$ if required, the same can be achieved for $V_2:=I-X_2^TY_2$. Let us now set $U_i:=I$ for $i=1,2$ to define invertible $\Yc_i,\,\Zc_i$ as in (\ref{fac}). Our choice ensures symmetry of $\Yc_i^T\Zc_i=\smat{Y_i^T&Y_i^TX_i+V_i^TU_i\\I&X_i}$ since $Y_i^TX_i+V_i^TU_i=I$. After a congruence transformation with $\Yc_i^{-1}$, this implies symmetry of $\Xc_i:=\Zc_i\Yc_i^{-1}$. Further, a similar argument shows $\Xc_1\cg0$.

	By inverting the congruence transformations at the end of Step 3, $Y_{2}\in \mathbf{P}_d$ and $X_{2}\in\mathbf{P}_p$ imply
	\eqref{diagb}.
Since $U_2$ and $V_2$ are invertible, we can render (\ref{exdec}) satisfied by taking $\hat{\De}_c$ with
\begin{equation}\label{scheduling}
X_2^T\De_lY_2+\De_l^T+U_2^T\hat{\De}_cV_2=0.
\end{equation}
This implies
\eqref{X2Dlc} and, hence, $\Xc_2\in\mathbf{P}$.
If solving \eqref{s2} for the controller matrices, we infer the validity of \eqref{pr2}.
Since $\Yc_1,\Yc_2$ are square and invertible, this ensures (\ref{pr1}), and thus (\ref{eqan2}).
\end{proof}

It is reassuring to extract that, after lifting, the extension of the known convexifying controller parameter transformation for nominal synthesis \cite{masubuchi1998}, \cite{scherer1997} turns out to be successful for gain-scheduling.
As a particular feature, the coupling between
the dynamics and the scheduling blocks causes no obstruction in this approach.
Moreover, as shown in the sufficiency proof, we can invert \eqref{s2} and \eqref{scheduling} to build the controller matrices and the scheduling function; the construction of the scheduling function turns out to be easier if compared to \cite{scherer2000}.
We stress again that, without lifting, such
a strategy miserably fails, and that no other gain-scheduling design technique without elimination and for full multipliers is available in the literature.

As a central next contribution of this work, we now show how the proposed lifting procedure is an essential enabler to design structured gain-scheduling for nested LPV systems.

\section{Gain-scheduling for structured interconnections}\label{secvi}

Let us now turn to the nested interconnections for partial and triangular gain-scheduling based on Fig.~\ref{fig0} and as described in the introduction. Motivated by Sec.~\ref{sech2}, we show how to embed these configurations into Fig.~\ref{figs}, only by introducing suitable structural requirements on the plant/controller LFR in (\ref{e1})/(\ref{e2}). For this purpose, let the uncertainty class $\mathbf{\De}$ with associated value set $\mathbf{V}$ be defined as in Sec.~\ref{sech2} such that some element $\De\in\mathbf{\De}$ takes values in $\mathbf{V}\subset\R^{r_\De\ti s_\De}$. Moreover,
we denote by $\RP$ the set of proper transfer matrices, and introduce the following compact notation for structured matrices.
\begin{defn}\label{struct_not}
	If some $2$-tuple $a=(a_1,a_2)\in\N_0^2$ has an zero component, the notational convention $(a_1,0)=a_1$ and $(0,a_2)=a_2$ is used in the sequel. For $a,\,b\in\N_0^2$,
	$$
	\T^{a\ti b}:=\left\{\smat{J_{11}&0\\J_{21}&J_{22}}\in\R^{a\ti b}\,\Big|\,J_{ij}\in\R^{a_i\ti b_j}\right\}
	$$
then denotes the set of lower triangular $2\ti2$ block-matrices
in a partition corresponding to $a$ and $b$, including the following natural versions for degenerate cases:
	$$
	\begin{aligned}
	\Tr^{a_1\ti b}&:=\T^{(a_1,0)\ti b},&\qquad\Tc^{a\ti b_2}&:=\T^{a\ti(0,b_2)},\\
	\Trf^{a_2\ti b}&:=\T^{(0,a_2)\ti b},&\qquad\Tcf^{a\ti b_1}&:=\T^{a\ti(b_1,0)}.\\[1ex]
	\end{aligned}
	$$
\end{defn}
\phantom{\vspace{1ex}}
\begin{exam} With Definition~\ref{struct_not} we have
	$$
	\mat{ccc;{2pt/1pt}cc;{2pt/1pt}c}{1&1&0&0&0&3\\0&1&1&2&2&3\\\hdl 4&4&0&5&5&0}
	\in\mat{lll}{\T^{(1,1)\ti (2,1)}&\Tc^{(1,1)\ti 2}&\Tcf^{(1,1)\ti 1}\\[1ex]\Tr^{1\ti (2,1)}&\R^{1\ti 2}&0}.
	$$
\end{exam}

\subsection{Classes $\Gc_2$, $\Kc_2$: Partial gain-scheduling}\label{classes1}
We consider Fig.~\ref{fig0} for partial gain-scheduling with the LPV system $P_1(\De)$ being described as the LFR
\begin{equation}\label{contr1}
\mat{c}{\hat{z}\\\xi\\y_1}=\mat{cc}{P_1^{11}&P_1^{12}\\P_1^{21}&P_1^{22}\\P_1^{31}&P_1^{32}}\mat{c}{\hat{w}\\u_1},\quad \hat{w}=\De\hat{z}
\end{equation}
where $u_1\to y_1$ is of size $k_1\ti m_1$ and $P_1^{ij}$ are LTI operators identified with their transfer matrices, i.e., $P_1^{ij}\in\RP$. If we partition $P_2\in\RP$ with $u_2\to y_2$ of size $k_2\ti m_2$ such that
$y_2=\smat{P_2^{11}&P_2^{12}}\smat{u_2\\\xi}$, the interconnection with (\ref{contr1}) yields
\begin{equation}\label{t1}
\mat{c}{\hat{z}\\y_1\\\hgl y_2}
\!=\!\mat{ccc}{P_1^{11}&P_1^{12}&0\\P_1^{31}&P_1^{32}&0\\\hgl P_2^{12}P_1^{21}&P_2^{12}P_1^{22}&P_2^{11}}\mat{c}{\hat{w}\\u_1\\ u_2}
,\ \hat{w}=\De\hat{z}.
\end{equation}
The lower-triangular structure results from the one-sided signal flow for $\xi$ in Fig.~\ref{fig0}. For well-posedness of the controller feedback loop, let us assume that $P_1^{22},\,P_1^{32},\,P_2^{11}$ are strictly proper.
By proceeding in a row-by-row fashion for the transfer matrix in (\ref{t1}), the latter assumption permits us to construct
state-space descriptions with the structure
$$
\mat{c}{\dot{x}_1\\\hline\hat{z}\\ y_1}=\matl({c|ccc}{\bu&\bu&\bu&0\\\hline \bu&\bu&\bu&0\\\bu&\bu&0&0})\mat{c}{x_1\\\hline \hat{w}\\ u_1\\ u_2}
,\ \
\mat{c}{\dot{x}_2\\\hline y_2}=\matl({c|ccc}{\bu&\bu&\bu&\bu\\\hline \bu&\bu&0&0})
\mat{c}{x_2\\\hline \hat{w}\\ u_1\\ u_2}
$$
and of McMillan degree $n_1$ and $n_2$, respectively. This leads to an overall state-space description of (\ref{t1}) with the structure
\begin{equation}\label{eqs1}
\mat{c}{\dot{x}_1\\\dot{x}_2\\\hline \hat{z}\\\hdl y_1\\ y_2}=
\mat{cc|c;{2pt/1pt}cc}{\bu&0&\bu&\bu&0\\0&\bu&\bu&\bu&\bu\\\hline
	\bu&0&\bu&\bu&0\\\hdl\bu&0&\bu&0&0\\ 0&\bu&\bu&0&0}
\mat{c}{x_1\\x_2\\\hline \hat{w}\\\hdl u_1\\ u_2},\quad \hat{w}=\De\hat{z}.
\end{equation}
Then, \eqref{eqs1} subsumes to the setting of Fig.~\ref{figs} after appending some performance channel $w_p\to z_p$ to \eqref{eqs1}, and if restricting the structure of the plant LFR \eqref{e1} for $G(\De)$ to
\begin{equation}\label{clp2}
\begin{aligned}
&\mat{cccc}{
	\Ah_{11}&\Ah_{12}&\Bh_1^p&\Bh_1\\
	\Ah_{21}&\Ah_{22}&\Bh_2^p&\Bh_2\\
	\Ch_1^p&\Ch_2^p&\Dh^p&\Eh\\
	\Ch_1&\Ch_2& \Fh&\Dh}
\in\mat{llll}{\T^{n\ti n}&\Tcf^{n\ti r}&\Tcf^{n\ti m^p}&\T^{n\ti m}\\\Tr^{s\ti n}&\R^{s\ti r}&\R^{s\ti m^p}& \Tr^{s\ti m}\\\Trf^{k^p\ti n}&\R^{k^p\ti r}&\R^{k^p\ti m^p}&\Trf^{k^p\ti m}\\\T^{k\ti n}&\Tcf^{k\ti r}&\Tcf^{k\ti m^p}&0}\\
&\text{and}\ \ \hat{\De}:\mathbf{V}\to\R^{r\times s},\ \hat{\De}(\De):=\De.
\end{aligned}
\end{equation}
Here we use $x:=\col(x_1,x_2)$, $u:=\col(u_1,u_2)$, $y:=\col(y_1,y_2)$ resulting in
\begin{equation}\label{dim}
n=\mat{cc}{n_1,n_2},\ \ m=\mat{cc}{m_1,m_2},\ \ k=\mat{cc}{k_1,k_2},
\end{equation}
and we view the block-diagonal sub-matrices in (\ref{eqs1}) as being lower triangular, while the definition of
$\hat{\De}(.)$ leads to the choice $r:=r_\De$, $s:=s_\De$.
Note that the performance channel $w_p\to z_p$ is determined by the underlying control problem in practice. Since this channel is not required to carry a particular structure in our approach, we can consider the associated matrices  $\Bh_i^p,\,\Ch_j^p,\,\Dh^p,\,\Eh,\,\Fh$ in \eqref{clp2} as being unstructured.

If translating the nested loop of $C_1(\De)$ and $C_2$ in Fig.~\ref{fig0} to the one in Fig.~\ref{figs} in an analogous fashion, we obtain the LFR for $K(\De)$ in (\ref{e2}), but now with the structural requirements
\begin{equation}\label{clk2}
\begin{aligned}
&\mat{ccc}{\Ah^c_{11}&\Ah^c_{12}&\Bh^c_1\\\Ah^c_{21}&\Ah^c_{22}&\Bh^c_2\\\Ch^c_1&\Ch^c_2&\Dh^c}\in\mat{ccc}{\T^{n^c\ti n^c}&\Tcf^{n^c\ti r^c}&\T^{n^c\ti k}\\\Tr^{s^c\ti n^c}&\R^{s^c\ti r^c}& \Tr^{s^c\ti k}\\\T^{m\ti n^c}&\Tcf^{m\ti r^c}&\T^{m\ti k}}\\
&\text{and}\ \ \hat{\De}_c:\mathbf{V}\to\R^{r^c\ti s^c},\ \De\mapsto\hat{\De}_c(\De)
\end{aligned}
\end{equation}
for the to-be-designed matrices with suitable block partitions $n^c\in\N_0^2$ and $r^c,\,s^c\in\N_0$ and some scheduling function $\hat{\De}_c(.)$. In the sequel, we compactly express the structural requirements in (\ref{clp2}) and (\ref{clk2}) by writing
$$
G(\De)\in\Gc_2\quad\text{and}\quad K(\De)\in\Kc_2.
$$

\subsection{Classes $\Gc_3$, $\Kc_3$: Triangular gain-scheduling}\label{classes2}

For triangular gain-scheduling, recall that we consider again
Fig.~\ref{fig0}, but now with $P_2=P_2(\De)$ and $C_2=C_2(\De)$ also depending on $\De\in\mathbf{\De}$.
If following the modeling procedure in Sec.~\ref{classes1}, this translates to Fig.~\ref{figs} with $G(\De)$ in (\ref{e1}) satisfying
\begin{equation}\label{clp3}
\begin{aligned}
&\mat{cccc}{
	\Ah_{11}&\Ah_{12}&\Bh_1^p&\Bh_1\\
	\Ah_{21}&\Ah_{22}&\Bh_2^p&\Bh_2\\
	\Ch_1^p&\Ch_2^p&\Dh^p&\Eh\\
	\Ch_1&\Ch_2& \Fh&\Dh}
\in\mat{llll}{\T^{n\ti n}&\T^{n\ti r}&\Tcf^{n\ti m^p}&\T^{n\ti m}\\\T^{s\ti n}&\T^{s\ti r}&\Tcf^{s\ti m^p}& \T^{s\ti m}\\\Trf^{k^p\ti n}&\Trf^{k^p\ti r}&\R^{k^p\ti m^p}&\Trf^{k^p\ti m}\\\T^{k\ti n}&\T^{k\ti r}&\Tcf^{k\ti m^p}&0}\\
&\text{and}\ \  \hat{\De}:\mathbf{V}\to\T^{r\times s},\ \hat{\De}(\De):=\mat{cc}{\De&0\\0&\De}
\end{aligned}
\end{equation}
with \eqref{dim}, where the size $r_\De\ti s_\De$ of $\De$ requires us to choose
\begin{equation}\label{dim2}
r=(r_1,r_2):=(r_\De,r_\De)\text{\ and\ }s=(s_1,s_2):=(s_\De,s_\De).
\end{equation}
Due to the $\De$-dependency in the outer loop of Fig.~\ref{fig0}, all matrices in (\ref{clp3}) except those related to the performance channel are triangular, while $\hat{\De}(.)$ is
diagonally repeated.

Further, $K(\De)$ in Fig.~\ref{figs} results from connecting $C_1(\De)$ with $C_2(\De)$. This is described in the state-space as in (\ref{e2}) for
\begin{equation}\label{clk3}
\begin{aligned}
&\hspace{-0.1ex}\mat{ccc}{\Ah^c_{11}&\Ah^c_{12}&\Bh^c_1\\\Ah^c_{21}&\Ah^c_{22}&\Bh^c_2\\\Ch^c_1&\Ch^c_2&\Dh^c}\in\mat{ccc}{\T^{n^c\ti n^c}&\T^{n^c\ti r^c}&\T^{n^c\ti k}\\\T^{s^c\ti n^c}&\T^{s^c\ti r^c}& \T^{s^c\ti k}\\\T^{m\ti n^c}&\T^{m\ti r^c}&\T^{m\ti k}}\\
&\hspace{-0.1ex}\arraycolsep.2ex\text{and}\
\hat{\De}_c\!:\!\mathbf{V}\!\to\T^{r^c\ti s^c},\
\hat{\De}_c(\De):=\mat{cc}{\hat{\De}^c_{11}(\De)&0\\\hat{\De}^c_{21}(\De)&\hat{\De}^c_{22}(\De)}
\end{aligned}
\end{equation}
with to-be-designed triangular matrices, partitions $n^c,\,r^c,\,s^c\in\N_0^2$ and the scheduling function $\hat{\De}_c(.)$.
Note that it is required to enforce the triangular structure of $\hat{\De}_c(.)$ in order to realize the one-sided controller communication structure in Fig.~\ref{fig0}.
In the sequel, we shortly refer to (\ref{clp3}) and (\ref{clk3}) as
$$
G(\De)\in\Gc_3\quad\text{and}\quad K(\De)\in\Kc_3.
$$

\subsection{Classes $\Gc_4$, $\Kc_4$: Zero feedthrough for $\Hz$-control}\label{classes3}

For $\Hz$-synthesis in Sec.~\ref{sech2} based on Fig.~\ref{figs}, finiteness of the closed-loop norm is only assured if the direct feedthrough term of $w_p\to z_p$ vanishes. However, since the initial closed-loop (\ref{eq3}) depends nonlinearly on $\De$, this cannot be systematically enforced by Theorem~\ref{theo:synthesis1}.

Our approach offers the following remedy. We assume that the direct feedthrough term of the LPV system $G(\Delta)$ is zero.
Following \cite{roesinger2020}, \cite{roesinger2019b}, one can then construct an LFR for $G(\De)$ in (\ref{e1}) which has a structured uncertainty channel according to
\begin{equation}\label{clp4}
\begin{aligned}
&\mat{cccc}{
	\Ah_{11}&\Ah_{12}&\Bh_1^p&\Bh_1\\
	\Ah_{21}&\Ah_{22}&\Bh_2^p&\Bh_2\\
	\Ch_1^p&\Ch_2^p&\Dh^p&\Eh\\
	\Ch_1&\Ch_2& \Fh&\Dh}
\in\mat{llll}{\R^{n\ti n}&\Trf^{n\ti r}&\R^{n\ti m^p}&\R^{n\ti m}\\\Tcf^{s\ti n}&\T^{s\ti r}& \Tc^{s\ti m^p}&\Tcf^{s\ti m}\\\R^{k^p\ti n}&\Tr^{k^p\ti r}&0&\R^{k^p\ti m}\\\R^{k\ti n}&\Trf^{k\ti r}&\R^{k\ti m^p}&0}\\
&\text{with}\ \ \hat{\De}:\mathbf{V}\to\T^{r\times s},\ \hat{\De}(\De):=\mat{cc}{\De&0\\0&\De}.
\end{aligned}
\end{equation}
Here, the partitions $r,\,s\in\N_0^2$ are again defined by \eqref{dim2} according to the dimension of $\De$. Similarly, one can enforce that the direct feedthrough term for the controller $K(\De)$ vanishes, by working
in (\ref{e2}) with the structured LFR
\begin{equation}\label{clk4}
\begin{aligned}
&\mat{ccc}{\Ah^c_{11}&\Ah^c_{12}&\Bh^c_1\\\Ah^c_{21}&\Ah^c_{22}&\Bh^c_2\\\Ch^c_1&\Ch^c_2&\Dh^c}\in\mat{ccc}{\R^{n^c\ti n^c}&\Trf^{n^c\ti r^c}&\R^{n^c\ti k}\\\Tcf^{s^c\ti n^c}&\T^{s^c\ti r^c}& \Tc^{s^c\ti k}\\\R^{m\ti n^c}&\Tr^{m\ti r^c}&0}\\
&\text{and}\ \
\hat{\De}_c:\mathbf{V}\to\T^{r^c\ti s^c},\
\hat{\De}_c(\De):=\mat{cc}{\hat{\De}^c_{11}(\De)&0\\\hat{\De}^c_{21}(\De)&\hat{\De}^c_{22}(\De)}
\end{aligned}
\end{equation}
for to-be-determined dimensions $n^c\in\N_0$, $r^c,\,s^c\in\N_0^2$ and a triangular scheduling block.
If closing the loops (\ref{e1}), (\ref{e2}) for (\ref{clp4}), (\ref{clk4}), a short calculation indeed shows
\begin{equation}\label{decp1}
	\mat{c}{\dot{x}\\\hline z_p\\ y}\,{=}\,\mat{c|cc}{\bu_\De&\bu_\De&\bu_\De\\\hline
		\bu_\De&0 &\bu_\De\\ \bu_\De &  \bu_\De&\bu_\De}\mat{c}{x\\\hline w_p\\ u},
	\mat{c}{\dot{x}_c\\\hline u}\,{=}\,\mat{c|c}{\bu_\De&\bu_\De\\\hline\bu_\De&0}\mat{c}{x_c\\\hline y}
\end{equation}
with $\De$-dependent matrices $\bu_\De$. Then it is guaranteed in a structural fashion that the
direct feedthrough term of $w_p\to z_p$ for the plant-controller interconnection (\ref{decp1})
vanishes identically. We compactly express (\ref{clp4}), (\ref{clk4}) by writing
$$
G(\De)\in\Gc_4\quad\text{and}\quad K(\De)\in\Kc_4.
$$

\section{Analysis for various interconnections}\label{sec3}

For all the plant/controller classes $\Gc_i$/$\Kc_i$ in Sec.~\ref{secvi}, we now present the precise problem formulation and the corresponding analysis conditions in parallel to Secs.~II.A - II.~C.
Also for the closed-loops with $\Gc_i$/$\Kc_i$, the $\Hz$-criterion can be
formulated as in Problem~\ref{problem1}.
Instead, to illustrate the flexibility to handle multiple objectives, the subsequent results are shown for $\Hi$-design with a bound on the $L_2$-gain as an objective for linear time-varying systems (see \cite{scherer2000} for a precise definition).

\begin{prob}\label{problem2}
	If $G(\De)\in\Gc_i$, find a controller $K(\De)\in\Kc_i$ such that the controlled LFR (\ref{eq3}) is well-posed, stable, and, for $x_e(0)=0$,  the squared $L_2$-gain of $w_p\to z_p$ is smaller than some given $\ga>0$ for all $\De\in\mathbf{\De}$.
\end{prob}

Let us formulate the corresponding analysis conditions with a compact yet insightful notation for matrix inequalities. If $X_1,\,X_2,\,Y$ are square matrices and $A_{ij},\,B_i,\,C_j,\,D$ are matrices with compatible dimensions for $i,j=1,2$, the expression
$$
\Lc\left(X_1,X_2,Y;\smat{A_{1j}&B_1\\A_{2j}&B_2\\C_j&D}\right)\cl0
$$
abbreviates the matrix inequality
\begin{equation}\label{notationineq}\arraycolsep.2ex
\mat{ccc}{I&0&0\\ A_{11}&A_{12}&B_1\\\hdl 0&I&0\\ A_{21}&A_{22}&B_2 \\\hdl 0&0&I\\C_1&C_2&D}^T\!\!
\mat{c;{2pt/1pt}c;{2pt/1pt}c}{X_1&0&0\\\hdl 0&X_2&0\\\hdl 0&0&Y}
\mat{ccc}{I&0&0\\ A_{11}&A_{12}&B_1\\\hdl 0&I&0\\ A_{21}&A_{22}&B_2 \\\hdl 0&0&I\\C_1&C_2&D}\cl0.
\end{equation}
The first three arguments of $\Lc(.)$ constitute the diagonal blocks of the middle matrix in (\ref{notationineq}), while the sub-matrices of the last argument are collected in the outer factor of (\ref{notationineq}) according to the induced partitions.
Here, $X_i$ and $Y$ are associated with the block rows $(A_{i1},A_{i2},B_i)$ and $(C_1,C_2,D)$ for $i=1,2$, respectively.
Note that the column partition and the dimensions of the identity matrices in the outer factor of \eqref{notationineq} are uniquely determined since $X_i$ and $Y$ are square. For example, if $(C_1,C_2)\in\R^{k\times\bullet}$ then $D$ has the dimension $k\times(\dim(Y)-k)$.

If the closed-loop scaling class $\mathbf{\hat{P}}$ is again defined as in (\ref{posconstr}),
the analogon of Theorem~\ref{theo:analysish2} for the $\Hi$-cost reads as follows.
\begin{thm}\label{theo:analysisu}
	Problem~\ref{problem2} is solved for $G(\De)\in\Gc_i$, $K(\De)\in\Kc_i$ if there exist $\Xc_1\cg0$ and $\hat{\Pc}\in\mathbf{\hat{P}}$ such that the associated closed-loop system (\ref{eq3}) satisfies
\begin{equation}\tag{$A\text{-}\Hi$}\label{AHi}
\Lc\left(\smat{0&\Xc_1\\\Xc_1&0},\hat{\Pc},\smat{-\ga I&0\\0&I};\smat{\hat\Ac_{1j}&\hat\Bc_1\\\hat\Ac_{2j}&\hat\Bc_2\\\hat\Cc_j&\hat\Dc}\right)\cl0.\\[1.5ex]
\end{equation}
\end{thm}

We emphasize that the analysis inequalities in Theorem~\ref{theo:analysish2} can as well be compactly expressed as
\begin{equation}\tag{$A\text{-}\Hz$}
	\begin{array}{l}
	\Lc\left(\smat{0&\Xc_1\\\Xc_1&0},\hat{\Pc},\smat{-\ga I&0\\0&0};\smat{\hat{\Ac}_{1j}&\hat{\Bc}_1\\\hat{\Ac}_{2j}&\hat{\Bc}_2\\\hat{\Cc}_j&\hat{\Dc}}\right)\cl0,\\[2ex]
	\Lc\left(\smat{-\Xc_1&0\\0&0},\hat{\Pc},\smat{0&0\\0&Z^{-1}};\smat{\hat{\Ac}_{1j}\\\hat{\Ac}_{2j}\\\hat{\Cc}_j}\right)\cl0.
	\end{array}
\end{equation}
This observation makes it possible to routinely translate all $\Hi$-synthesis results into the corresponding versions for the $\Hz$-cost, which is omitted for reasons of space.

It is crucial to recognize that our notational preparation in Sec.~\ref{seclifting} allows us
to easily apply the lifting technique for all structured plant and controller pairs
$\Gc_i/\Kc_i$. The lifted plant LFR for $\Gc_i$ is again given by (\ref{liftedblock}), (\ref{eq5})
and abbreviated as $\Gc^l_i$ analogously to Definition~\ref{liftedLFR}.
Notice that
the respective structural properties of $\Gc_i$ are inherited by $\Gc_i^l$.
However, as motivated in Sec.~\ref{seclifting}, we {\em do not lift the controller.} Instead, we
design controllers with their required original structure, but with a square scheduling
block of dimension $|l^c|$, for the respective lifted system $\Gc_i^l$.
Note that the  analysis conditions for the associated closed-loop LFR (\ref{eq3lifted}), (\ref{delc})
are obtained by simply dropping the hats of the matrices in the outer factors
of $(A\text{-}\Hi)$, and by replacing $\hat{\Pc}\in\mathbf{\hat{P}}$ with
$\smat{0&\Pc\\\Pc&0}$ for $\Pc\in\mathbf{P}$.

This strategy leads to a valid controller for the original unlifted plant $\Gc_i$ as well,
which is formulated as follows, and proved in parallel to Theorem~\ref{theo:analysislh2}.
\begin{thm}\label{theo:analysisl}
If there exists a structured controller $K(\De)\in\Kc_i$ with $l^c=r^c=s^c$, $\Xc_1\cg0$ and $\Pc\in\mathbf{P}$, such that the closed-loop system (\ref{eq3lifted}), (\ref{delc}) for $G_l(\De)\in\Gc_i^l$ satisfies
\begin{equation}\tag{$L\text{-}\Hi$}\label{LHi}
\Lc\left(\smat{0&\Xc_1\\\Xc_1&0},\smat{0&\Pc\\\Pc&0},\smat{-\ga I&0\\0&I};\smat{\Ac_{1j}&\Bc_1\\\Ac_{2j}&\Bc_2\\\Cc_j&\Dc}\right)\cl0,
\end{equation}
then $(A\text{-}\Hi)$  holds for the closed-loop system (\ref{eq3}) obtained for
$G(\De)\in\Gc_i$ interconnected with the very same controller $K(\De)\in\Kc_i$ and for some full block scaling $\hat{\Pc}\in\mathbf{\hat{P}}$.
\end{thm}

In Sec.~\ref{secsyn} we clarify that lifting is, once again, the key enabler
for the convexification of synthesizing structured controllers for structured plants.

\section{A framework for gain-scheduled synthesis}\label{secsyn}

Let us develop our novel synthesis framework for all classes $\Gc_i/\Kc_i$ from Sec.~\ref{secvi}.
Motivated by the unstructured $\Hz$-setting in Sec.~\ref{syn4},
we expect the lifted analysis conditions $(L\text{-}\Hi)$ in Theorem~\ref{theo:analysisl}
to be matched to the $\Hi$-synthesis conditions
\begin{equation}\tag{$S\text{-}\Hi$}\label{SHi}
\Lc\left(\smat{0&I\\I&0},\smat{0&I\\I&0},\smat{-\ga I&0\\0&I};\smat{\Ahb_{1j}&\Bhb_1\\\Ahb_{2j}&\Bhb_2\\\Chb_j&\Dhb}\right)\cl0\text{\ \ and\ \ }\Xhb\cg0,
\end{equation}
where $\Ahb_{ij},\,\Bhb_i,\,\Chb_j,\,\Dhb$ and $\Xhb$ depend affinely on some new decision variables.
We first show the result for triangular gain-scheduling with $\Kc_3$ in Sec.~\ref{syn2}, which is the main synthesis result in this paper. This will be specialized to $\Kc_2,\,\Kc_4$ in Secs.~\ref{syn1} and \ref{syn3}, respectively.
This procedure has the advantage to exhibit the complexity of the designs as reflected in the size and structure of the decision variables.
For this purpose, we employ the following projection operators in relation to the subspaces $\T,\,\Tr$ in Definition~\ref{struct_not}.
\begin{defn}\label{defproj}
	For partitioned matrices $\smat{J_{11}&J_{12}\\J_{21}&J_{22}}$, $\smat{J_1&J_2}$, we define the projection $\Lo(.)$ onto their lower part, $\U(.)$ onto their strict upper part, and $\Ur$ onto their right part as
	$$
	\begin{aligned}
	\Lo\smat{J_{11}&J_{12}\\J_{21}&J_{22}}&:=\smat{J_{11}&0\\J_{21}&J_{22}},
	&\U\smat{J_{11}&J_{12}\\J_{21}&J_{22}}&:=\smat{0&J_{12}\\0&0},\\
	\Ur\smat{J_1&J_2}&:=\smat{0&J_2}.
	\end{aligned}
	$$
\end{defn}
\phantom{space}

In the sequel, we employ several structured unknowns consisting of decision variables and constant blocks.
To display the convexification mechanism, we use boldface notations for the variable parts to differentiate
them from those matrices which are constant; recall from Sec.~\ref{sech2} that matrices with a hat
might carry a specific sub-structure by themselves.

We close this chapter by showing possible extensions of our framework in Sec.~\ref{extensions}, while discussing the question of optimality in the lifted setting in Sec.~\ref{secopt}.

\subsection{Synthesis for $\Gc_3$, $\Kc_3$: Triangular gain-scheduling}\label{syn2}

Let us motivate the choice of the design variables by putting an emphasis on the
relevant modifications if compared to unstructured gain-scheduling in Sec.~\ref{syn4}.
For triangular gain-scheduling with $\Gc_3/\Kc_3$, we have to structurally restrict
the plant/controller LFRs (\ref{e1})/(\ref{e2}) to \eqref{clp3}/\eqref{clk3}, which involves
triangular matrices both in the control and uncertainty channel. In the subsequent construction, we take inspiration
from the nominal triangular controller design procedure in \cite{scherer2014}, \cite{roesinger2019a}. For compact notations, let us introduce the partitions
\begin{equation}\label{partitionex}
\begin{aligned}
&\hspace{-1ex}a=(a_1,a_2):=(|n|,|n|),&\,b&=(b_1,b_2):=(|l|,|l|),\\
&\hspace{-1ex}u=(u_1,u_2):=(|n|,|n|),&\,v&=(v_1,v_2):=(|l|,|l|).
\end{aligned}
\end{equation}
Instead of \eqref{v1}, we now choose the decision variables
\begin{equation}\label{tgs1}
X_1:=\mat{cc}{\bs{X_2}&\bs{X_3}}\in\R^{n\ti a},\quad
Y_1:=\mat{cc}{\bs{Y_1}&\bs{Y_2}}\in\R^{n\ti u}
\end{equation}
with unstructured symmetric matrices $\bs{X_3},\,\bs{Y_1}\in\Sb^{n}$ and sub-structured coupled blocks
\begin{equation}\label{ev1}
\begin{aligned}
\bs{X_2}&:=\mat{c}{\bs{\hat{X}_2}\\\hat{X}_2}:=\mat{cc}{\bs{X_{22}}&\bs{Z}^T_{\bs{22}}\\0&I_{n_2}}\in\R^{n\ti n},\\
\bs{Y_2}&:=\mat{c}{\hat{Y}_2\\\bs{\hat{Y}_2}}:=
\mat{cc}{I_{n_1}&0\\-\bs{Z_{22}}&\bs{Y_{22}}}\in\R^{n\ti n},
\end{aligned}
\end{equation}
where $\bs{X_{22}}\in\Sb^{n_1}$, $\bs{Y_{22}}\in\Sb^{n_2}$ are symmetric and the shared variable
$\bs{Z_{22}}\in\R^{n_2\ti n_1}$ is general. Note that the latter partitions are motivated by
the $2\ti 2$-block triangular structure of the state-matrix $A_{11}\in\T^{n\ti n}$ of the lifted plant $G_l(\De)\in\Gc_3^l$,
as inherited from the one of the original system $G(\De)\in\Gc_3$. For the quick identification of structural dependencies,
bold and non-bold notations are used for sub-matrices according to whether they are genuine variables or fixed, respectively.
In the synthesis conditions, \eqref{s22c} is replaced by
\begin{equation}\label{couplingt}
\Xhb:=\mat{ccc}{\bs{Y_1}&\bs{Y_2}&I\\\bs{Y}^T_{\bs{2}}&\bs{Y}^T_{\bs{2}}\bs{X_2}&\bs{X}^T_{\bs{2}}\\I&\bs{X_2}&\bs{X_3}}\in\Sb^{(u_1,u_2,a_2)},
\end{equation}
where it is noted that
$\bs{Y}^T_{\bs{2}}\bs{X_2}=\smat{\hat{Y}_2\\\bs{\hat{Y}_2}}^T\smat{\bs{\hat{X}_2}\\\hat{X}_2}=\smat{\bs{X_{22}}&0\\0&\bs{Y_{22}}}$ is symmetric and depends affinely on the decision variables.
Therefore, $\Xhb\cg0$ in \eqref{SHi} does indeed constitute an LMI.

Let us now introduce a novel collection of decision variables that correspond to the multiplier $\Pc$ in \eqref{LHi} and is adapted to the lifted LFR. Recall that the uncertainty channel of the initial plant LFR \eqref{clp3} comprises the triangular matrix $\Ah_{22}\in\T^{s\ti r}$. This allows us to express $A_{22}$ in the lifted LFR \eqref{eq5} as
\begin{equation}\label{Apartition}
A_{22}=\mat{cc}{I_{r}&0\\2\Ah_{22}&-I_{s}}=\mat{cc;{2pt/1pt}cc}{I_{r_1}&0&0&0\\0&I_{r_2}&0&0\\\hdl \bu&0&-I_{s_1}&0\\\bu&\bu&0&-I_{s_2}},
\end{equation}
a $2\ti 2$ block triangular matrix with a refined sub-partition
	\begin{equation}\label{Lpartition}
	l=(r,s)=((r_1,r_2),(s_1,s_2)).
	\end{equation}
Analogously to the doubling of the column length in \eqref{tgs1} if compared to the unstructured case, the $2\ti2$ block-triangular structure of $A_{22}$ in \eqref{Apartition} motivates to also double the column length $|l|$ of the multiplier variables \eqref{v2} in unstructured synthesis. Recalling \eqref{partitionex}, we hence pick
\begin{equation}\label{tgs2}
X_2:=\mat{cc}{\bs{P_2}&\bs{P_3}}\in\R^{l\ti b},\quad
Y_2:=\mat{cc}{\bs{Q_1}&\bs{Q_2}}\in\R^{l\ti v}
\end{equation}
with $l\ti l$-dimensional sub-blocks $\bs{P_i},\,\bs{Q_j}$. As in Sec.~\ref{syn4}, we impose the constraints $\bs{P_3}\in\mathbf{P}_p$ and $\bs{Q_1}\in\mathbf{P}_d$ with the primal/dual scaling class (\ref{primal})/(\ref{dual}), but with the difference that $\mathbf{P}_p/\mathbf{P}_d$ are now
defined for $\Gc_3$ in relation to the structured uncertainty block $\hat{\De}(\De)=\smat{\De&0\\0&\De}$. In view of the hierarchical structure of $A_{22}$ in \eqref{Apartition}, we use the structured square variables
\begin{equation}\label{ev3}
\begin{aligned}
\bs{P_2}&:=\mat{c}{\bs{\hat{P}_{2}}\\\hat{P}_{2}\\\bs{\hat{P}_{3}}\\\hat{P}_{3}}:=
	\mat{cc;{2pt/1pt}cc}{\bs{P_{22}}&\bs{R}^T_{\bs{22}}&\bs{P_{23}}&\bs{R}^T_{\bs{32}}\\ 0&I_{r_2}&0&0\\\hdl \bs{P_{32}}&\bs{R}^T_{\bs{23}}&\bs{P_{33}}&\bs{R}^T_{\bs{33}}\\0&0&0&I_{s_2}}
\in\R^{l\ti l}
,\\
\bs{Q_2}&:=\mat{c}{\hat{Q}_2\\\bs{\hat{Q}_2}\\\hat{Q}_3\\\bs{\hat{Q}_3} }
:=
\mat{cc;{2pt/1pt}cc}{I_{r_1}&0&0&0\\-\bs{R_{22}}&\bs{Q_{22}}&-\bs{R_{23}}&\bs{Q_{23}}\\\hdl 0&0&I_{s_1}&0\\-\bs{R_{32}}&\bs{Q_{32}}&-\bs{R_{33}}&\bs{Q_{33}}}
\in\R^{l\ti l}
\end{aligned}
\end{equation}
according to $l$ and its sub-partition in \eqref{Lpartition}, where
\begin{equation}\label{ev2}
\begin{aligned}
&\mat{cc}{\bs{P_{22}}&\bs{P_{23}}\\\bs{P_{32}}&\bs{P_{33}}}\,\in\Sb^{(r_1,s_1)},\quad
\mat{cc}{\bs{Q_{22}}&\bs{Q_{23}}\\\bs{Q_{32}}&\bs{Q_{33}}}\in\Sb^{(r_2,s_2)},\\
&\mat{cc}{\bs{R_{22}}&\bs{R_{23}}\\\bs{R_{32}}&\bs{R_{33}}}\in\R^{(r_2,s_2)\ti(r_1,s_1)}.
\end{aligned}
\end{equation}
As a counterpart to $\Xhb\cg0$ for \eqref{couplingt} involving $\bs{Y}^T_{\bs{2}}\bs{X_2}$, the synthesis conditions also involve the multiplier constraint
\begin{equation}\label{couplings}
\He\left[\bs{P}^T_{\bs{2}}\De_l(\De)\bs{Q_2}\right]\cg0\quad\text{for all}\quad \De\in\mathbf{V}.
\end{equation}
Note that \eqref{ev3} and \eqref{ev2} are coupled similarly to \eqref{ev1}, but now with a specific sub-structure having the key benefit that \eqref{couplings} is actually affine in the decision variables. Indeed, with the definition (\ref{liftedblock}) of $\De_l(\De)$ involving $\hat{\De}(\De)$ for $\Gc_3$, this is extracted from
$$
\begin{aligned}
&\hspace{-1.5ex}\bs{P}^T_{\bs{2}}\De_l(\De)\bs{Q_2}=\bs{P}^T_{\bs{2}}\smat{-I_{r}&2\hat{\De}(\De)\\0&I_{s}}\bs{Q_2}=\\
&\hspace{-1.5ex}=\smat{\bs{\hat{P}_{3}}\\\hat{P}_{3}}^T\smat{\hat{Q}_3\\\bs{\hat{Q}_3}}-\smat{\bs{\hat{P}_{2}}\\\hat{P}_{2}}^T\smat{\hat{Q}_2\\\bs{\hat{Q}_2}}+
2\smat{\bs{\hat{P}_{2}}\\\hat{P}_{2}}^T\!\!\smat{\De&0\\0&\De}\smat{\hat{Q}_3\\\bs{\hat{Q}_3}}
\end{aligned}
$$
and our convention to differentiate bold from non-bold blocks.

As in Sec.~\ref{syn4}, our design relies on a convexifying parameter transformation matching the controller matrices to new decision variables. As seen for nominal triangular design \cite{scherer2014}, \cite{roesinger2019a}, such a transformation can be made applicable to preserve triangular structures.
Based on the choice \eqref{tgs1}-\eqref{ev1}, \eqref{tgs2}-\eqref{ev2},
our lifting approach even shows this fact for triangular gain-scheduling.
To this end, we match the triangular controller matrices \eqref{clk3} for $\Kc_3$ to the triangular unknowns
\begin{equation}\label{klmn}
\mat{ccc}{\bs{K_{11}}&\bs{K_{12}}&\bs{L_1}\\\bs{K_{21}}&\bs{K_{22}}&\bs{L_2}\\\bs{M_1}&\bs{M_2}&\bs{N}}\in\mat{lll}{\T^{a\ti u}&\T^{a\ti v}&\T^{a\ti k}\\\T^{b\ti u}&\T^{b\ti v}& \T^{b\ti k}\\\T^{m\ti u}&\T^{m\ti v}&\T^{m\ti k}}
\end{equation}
with the partitions defined by \eqref{partitionex}.

Lastly, for a compact exposition and by using  \eqref{partitionex}, we introduce the structured matrices
\begin{equation}\label{a2}
\begin{aligned}
B_1^e&:=\mat{cc}{0_{a_1\ti m_1}&\t B_1\\0_{a_2\ti m_1}&0_{a_2\ti m_2}},
&B_2^e&:=\mat{cc}{0_{b_1\ti m_1}&\t B_2\\0_{b_2\ti m_1}&0_{b_2\ti m_2}},\\
C_1^e&:=\mat{cc}{0_{k_1\ti u_1}&\t C_1\\0_{k_2\ti u_1}&0_{k_2\ti u_2}},
&C_2^e&:=\mat{cc}{0_{k_1\ti v_1}&\t C_2\\0_{k_2\ti v_1}&0_{k_2\ti v_2}};
\end{aligned}
\end{equation}
in here, the blocks $\t B_i$ and $\t C_j$ are extracted from $B_1\in\T^{n\ti m}$, $B_2\in\T^{l\ti m}$ and $C_1\in\T^{k\ti n}$, $C_2\in\T^{k\ti l}$ in the lifted LFR (\ref{eq5}) according to
\begin{equation}\label{a1}
\mat{c;{2pt/1pt}c}{\bu&\t B_i}:=\mat{c;{2pt/1pt}c}{\bu&0\\\bu&\bu}=B_i \te{and}
\mat{c}{\t C_j\\\hdl \bu}:=\mat{cc}{\bu&0\\\hdl \bu&\bu}=C_j.
\end{equation}
The main result for triangular gain-scheduling then reads as follows.
\begin{thm}\label{theo:synthesis4}
	There exists a structured controller $K(\De)\in\Kc_3$ with $l^c=r^c=s^c$ and $\Xc_1\cg0$, $\Pc\in\mathbf{P}$, such that the closed-loop system \eqref{eq3lifted}, \eqref{delc} obtained for $G_l(\De)\in\Gc_3^l$ satisfies \eqref{LHi} iff there exists Lyapunov variables \eqref{tgs1}-\eqref{ev1} specifying \eqref{couplingt}, multiplier variables \eqref{tgs2}-\eqref{ev2} and controller variables \eqref{klmn}
fulfilling \eqref{SHi} and \eqref{couplings} with $\Ahb_{ij},\,\Bhb_i,\,\Chb_j,\,\Dhb$ defined by \eqref{s22} for $N:=\bs{N}$ and
	\begin{equation}\label{klmn2}
	\begin{aligned}
	L_i&:=\bs{L_i}+\U(B_i^e\bs{N}),\  \ M_j:=\bs{M_j}+\U(\bs{N}C_j^e),\\
	K_{ij}&:=\bs{K_{ij}}+\U(X^T_iA_{ij}Y_j)+\\
	&\phantom{=}+\U(B_i^e\bs{M_j})+\U(\bs{L_i}C_j^e)-B_i^e\bs{N}C_j^e.\\
	\end{aligned}
	\end{equation}
	If \eqref{SHi} and \eqref{couplings} are satisfied, one can construct $\Xc_1\cg0$, $\Pc\in\mathbf{P}$, and $K(\De)\in\Kc_3$ with an affine triangular scheduling function $\hat{\De}_c(\De)$ such that \eqref{LHi} holds, while the size of the state and scheduling matrix $\Ah_{11}^c$ and $\Ah_{22}^c$ of the controller $K(\De)$ are bounded by $2|n|\ti 2|n|$ and $2|l|\ti 2|l|$, respectively.
\end{thm}

The constructive proof of Theorem~\ref{theo:synthesis4} is given in the Appendix.
In the sufficiency part, we derive explicit formulas for the controller matrices (Step~3) and for $\hat{\De}_c(\De)$ (Step~2).

Since $\hat{\De}(.)$ is affine and $\mathbf{V}=\text{Co}\{\De_1,\ldots,\De_N\}$, recall that
the constraints $\bs{P_3}\in\mathbf{P}_p$, $\bs{Q_1}\in\mathbf{P}_d$ admit an LMI representation.
Similarly, \eqref{couplings} can be equivalently replaced by $\He\left[\bs{P}^T_{\bs{2}}\De_l(\De)\bs{Q_2}\right]\cg0$ for all $\Delta\in\{\De_1,\ldots,\De_N\}$. Hence, Theorem~\ref{theo:synthesis4} leads to a finite dimensional LMI test, since all variables enter in an affine fashion as is easily checked also for those expressions that have not yet been considered. E.g.,
$$
\begin{aligned}
&\hspace{6ex}\U(X^T_1A_{12}Y_2)=\smat{0&\bs{X}^T_{\bs{2}}A_{12}\bs{Q_2}\\0&0}\te{with}\\
&\bs{X}^T_{\bs{2}}A_{12}\bs{Q_2}=\bs{X}^T_{\bs{2}}\mat{cc}{\Ah_{12}&0_{n\ti s}}\bs{Q_2}
=\smat{\bs{\hat{X}_{2}}\\\hat{X}_{2}}^T\Ah_{12}\smat{\hat{Q}_2\\\bs{\hat{Q}_{2}}}
\end{aligned}
$$
is indeed affine in $\bs{\hat{X}_{2}}$ and $\bs{\hat{Q}_{2}}$ since $\Ah_{12}\in\T^{n\ti r}$ is triangular.

Let us finally emphasize that by now standard relaxation strategies
permit us to easily extend our results to value sets $\mathbf{V}$ that are semi-algebraic, for example \cite{Sch06,VeeSch16a}.

\subsection{Synthesis for $\Gc_2$, $\Kc_2$: Partially scheduled nested loop}\label{syn1}

For $\Gc_2$/$\Kc_2$ we consider the plant/controller LFRs in (\ref{e1})/(\ref{e2}) with \eqref{clp2}/\eqref{clk2}, respectively. As for $\Gc_3$/$\Kc_3$, all matrices in the control channels are $2\ti2$ block triangular. In contrast, the uncertainty channels carry no partition as in the unstructured design with $\Gc_1$/$\Kc_1$.
This motivates to choose the structured design variables for the Lyapunov part as in Sec.~\ref{syn2}, i.e., $X_1,\,Y_1$ as in \eqref{tgs1}-\eqref{ev1} with $\Xhb$ in \eqref{couplingt}. On the other hand, we rely for the scaling part on the choice in Sec.~\ref{syn4}, i.e.,
\begin{equation}\label{v22}
X_2:=\bs{P_2}\in \mathbf{P}_p,\qquad Y_2:=\bs{Q_1}\in\mathbf{P}_d
\end{equation}
with the primal/dual scaling class (\ref{primal})/(\ref{dual}) for $\hat{\De}(\De)=\De$.

In a fashion analogous to the design results for $\Kc_3$, $\Kc_1$, we take controller variables that match the structure of $\Kc_2$, i.e.,
\begin{equation}\label{v7}
\mat{ccc}{\bs{K_{11}}&\bs{K_{12}}&\bs{L_1}\\\bs{K_{21}}&\bs{K_{22}}&\bs{L_2}\\\bs{M_1}&\bs{M_2}&\bs{N}}\in\mat{lll}{\T^{a\ti u}&\Tcf^{a\ti |l|}&\T^{a\ti k}\\\Tr^{|l|\ti u}&\R^{|l|\ti |l|}& \Tr^{|l|\ti k}\\\T^{m\ti u}&\Tcf^{m\ti |l|}&\T^{m\ti k}}
\end{equation}
with $a,\,u\in\N_0^2$ as in \eqref{partitionex}.
The next result with $B_1^e,\,C_1^e$ from \eqref{a2} can be shown by directly simplifying the proof of Theorem~\ref{theo:synthesis4}; we drop the details for brevity.
\begin{thm}\label{theo:synthesis3}
	There exists a structured controller $K(\De)\in\Kc_2$ with $l^c=r^c=s^c$ and $\Xc_1\cg0$, $\Pc\in\mathbf{P}$, such that the closed-loop system (\ref{eq3lifted}), (\ref{delc}) obtained for $G_l(\De)\in\Gc_2^l$ satisfies \eqref{LHi} iff there exists Lyapunov variables \eqref{tgs1}-\eqref{ev1} specifying \eqref{couplingt}, multiplier variables \eqref{v22} and controller variables \eqref{v7} fulfilling \eqref{SHi} with $\Ahb_{ij},\,\Bhb_i,\,\Chb_j,\,\Dhb$ defined by (\ref{s22}) for $N:=\bs{N}$, for \eqref{klmn2} if $(i,j)=(1,1)$, and for
	\begin{equation}\label{klmn22}
	\begin{aligned}
	L_2&:=\bs{L_2},\hspace{0.75ex}  M_2:=\bs{M_2},\hspace{0.75ex}  K_{12}:=\bs{K_{12}},\hspace{0.75ex}  K_{22}:=\bs{K_{22}},\\
	K_{21}&:=\bs{K_{21}}+\Ur(X_2^TA_{21}Y_1)+\Ur(\bs{L_2}C^e_1).
	\end{aligned}
	\end{equation}
	If \eqref{SHi} is satisfied, one can construct $\Xc_1\cg0$, $\Pc\in\mathbf{P}$, and $K(\De)\in\Kc_2$ with an affine scheduling function $\hat{\De}_c(\De)$ such that \eqref{LHi} holds, while the size of the state and scheduling matrix $\Ah_{11}^c$ and $\Ah_{22}^c$ of the controller $K(\De)$ are bounded by $2|n|\ti 2|n|$ and $|l|\ti |l|$, respectively.
\end{thm}

Let us now clarify that \eqref{klmn22} is a direct specialization of \eqref{klmn2} for fully triangular synthesis.
Indeed, the controller variables \eqref{klmn} specialize to \eqref{v7} with
\begin{equation}\label{d2}
b=(b_1,b_2):=(|l|,0)\quad\text{and}\quad v=(v_1,v_2):=(|l|,0).
\end{equation}
Moreover, since all blocks in \eqref{klmn2} inherit the partition of \eqref{klmn}, $v_2=0$ implies $\U(.)=0$ for $K_{12},\,K_{22}$ and $M_2$, while $b_2=0$ shows $\U(.)=\Ur(.)$ for $K_{21}$ and $L_2$. In \eqref{a2}, we note that $C^e_2=0$ due to \eqref{d2}. We also have $B^e_2=0$,
since the initial LFR for $\Gc_2$ involves $\Bh_2\in\Tr^{s\ti m}$, which implies $B_2\in\Tr^{l\ti m}$ for the lifted LFR \eqref{eq5} and thus $\t{B}_2=0$ in \eqref{a1}.

Also note that \eqref{clp3} for $\Gc_3$ specializes to \textbf{\eqref{clp2}} for $\Gc_2$ with $r_2=s_2=0$.
The latter implies $\bs{P_2}=\smat{\bs{P_{22}}&\bs{P_{23}}\\\bs{P_{32}}&\bs{P_{33}}}$ and $\bs{Q_2}=I_{(r_1,s_1)}$ in \eqref{ev3}. Due to \eqref{couplings}, this shows $\bs{P_2}\in\mathbf{P}_p$ in correspondence with \eqref{v22}.

Hence, the partial gain-scheduling problem can be viewed as located between triangular and unstructured gain-scheduling, with the bound for the dimension $|l^c|$ of the scheduling block being reduced from $2|l|$ to $|l|$ if compared to Theorem~\ref{theo:synthesis4}.

\subsection{Synthesis for $\Gc_4$, $\Kc_4$: Zero feedthrough term}\label{syn3}

The plant/controller LFRs (\ref{e1})/(\ref{e2}) are given for $\Gc_4$/$\Kc_4$ by \eqref{clp4}/\eqref{clk4}.
In contrast to Sec.~\ref{syn1}, the control channel carries no particular structure, while the scheduling channel involves (truncated) triangular matrices. Hence, let us take variables for the Lyapunov part as for unstructured design in Sec.~\ref{syn4}, i.e.,
\begin{equation}\label{v12}
X_1:=\bs{X_3}\in\Sb^{n},\qquad Y_1:=\bs{Y_1}\in\Sb^{n},
\end{equation}
while using \eqref{s22c} for $\Xhb$. For the scaling part, this motivates to define $X_2,\,Y_2$ as in Sec.~\ref{syn2}, i.e., the structured variables \eqref{tgs2}-\eqref{ev2} with coupling condition \eqref{couplings}. In analogy to $\Kc_1,\,\Kc_2$ and $\Kc_3$, let us pick the controller variables
\begin{equation}\label{v5}
\mat{ccc}{\bs{K_{11}}&\bs{K_{12}}&\bs{L_1}\\\bs{K_{21}}&\bs{K_{22}}&\bs{L_2}\\\bs{M_1}&\bs{M_2}&\bs{N}}\in\mat{lll}{\R^{|n|\ti |n|}&\Trf^{|n|\ti v}&\R^{|n|\ti k}\\\Tcf^{b\ti |n|}&\T^{b\ti v}& \Tc^{b\ti k}\\\R^{m\ti |n|}&\Tr^{m\ti v}&0}
\end{equation}
for $b,\,v\in\N_0^2$ in \eqref{partitionex} to reflect the sparsity structure of $\Kc_4$.
\begin{thm}\label{theo:synthesis2}
	There exists a structured controller $K(\De)\in\Kc_4$ with $l^c=r^c=s^c$, and $\Xc_1\cg0$, $\Pc\in\mathbf{P}$, such that the closed-loop system (\ref{eq3lifted}), (\ref{delc}) obtained for $G_l(\De)\in\Gc_4^l$ satisfies \eqref{LHi}
	iff there exist Lyapunov variables (\ref{v12}) specifying \eqref{s22c}, multiplier variables \eqref{tgs2}-\eqref{ev2}, and controller variables (\ref{v5}) fulfilling \eqref{SHi} and \eqref{couplings} with  $\Ahb_{ij},\,\Bhb_i,\,\Chb_j,\,\Dhb$ defined by (\ref{s22}) for $N:=\bs{N}$ and
	\begin{equation}\label{v6}
	\begin{aligned}
	L_i&:=\bs{L_i},\ \ M_j:=\bs{M_j},\ \ K_{ij}:=\bs{K_{ij}}\ \text{if}\ (i,j)\neq(2,2),\\
	K_{22}&:=\bs{K_{22}}+\U(X_2^TA_{22}Y_2).
	\end{aligned}
	\end{equation}
	If \eqref{SHi} is satisfied, one can construct $\Xc_1\cg0$, $\Pc\in\mathbf{P}$, and $K(\De)\in\Kc_4$ with an affine triangular scheduling function $\hat{\De}_c(\De)$ such that \eqref{LHi} holds, while the size of the state and scheduling matrix $\Ah_{11}^c$ and $\Ah_{22}^c$ of the controller $K(\De)$ are bounded by $|n|\ti |n|$ and $2|l|\ti 2|l|$, respectively.
\end{thm}
Analogously to Sec.~\ref{syn1}, \eqref{klmn2} specializes to \eqref{v6}. Indeed, with
$$
\begin{aligned}
&a=(a_1,a_2):=(0,|n|), &\,u&=(u_1,u_2):=(|n|,0),\\
&k=(k_1,k_2)=(0,k_2), &\,m&=(m_1,m_2)=(m_1,0),
\end{aligned}
$$
we note that the controller variables \eqref{klmn} reduce to \eqref{v5}, we get $B^e_i=0$ and $C^e_j=0$ in \eqref{a2}, and $\U(.)=0$ in \eqref{klmn2} except for $K_{22}$.
The proof of Theorem~\ref{theo:synthesis2} is omitted for brevity since it proceeds along the lines of the one for Theorem~\ref{theo:synthesis4}. Instead, we emphasize that the unstructured control channel for $\Gc_4/\Kc_4$ allows to perform synthesis with a controller of McMillan degree at most $|n|$, if compared to $2|n|$ in Theorem~\ref{theo:synthesis4}.

\subsection{Direct extensions}\label{extensions}

The compact presentation of our framework allows to easily combine results. For instance, if we aim to solve the triangular $\Hz$-gain-scheduling problem, while simultaneously guaranteeing a finite closed-loop $\Hz$-norm, we use the $\Hz$-analogon  of the synthesis inequalities in \eqref{SHi}, i.e.
\begin{equation}\tag{$S\text{-}\Hz$}\label{synin}
\begin{aligned}
&\Lc\left(\smat{0&I\\I&0},\smat{0&I\\I&0},\smat{-\ga I&0\\0&0};\smat{\Ahb_{1j}&\Bhb_1\\\Ahb_{2j}&\Bhb_2\\\Chb_j&\Dhb}\right)\cl0,\\
&\Lc\left(\smat{-\Xhb&0\\0&0},\smat{0&I\\I&0},\smat{0&0\\0&Z^{-1}};\smat{\Ahb_{1j}\\\Ahb_{2j}\\\Chb_{j}}\right)\cl0,
\end{aligned}
\end{equation}
while restricting the plant/controller LFRs (\ref{e1})/(\ref{e2}) to $\Gc_3\cap\Gc_4$/$\Kc_3\cap\Kc_4$.
For instance, $\Kc_3\cap \Kc_4$ means to work with controller matrices that are structured as in
$$
\mat{ccc}{\Ah^c_{11}&\Ah^c_{12}&\Bh^c_1\\\Ah^c_{21}&\Ah^c_{22}&\Bh^c_2\\\Ch^c_1&\Ch^c_2&\Dh^c}\in\mat{ccc}{\T^{n^c\ti n^c}&\T^{n^c\ti r^c}&\T^{n^c\ti k}\\\T^{s^c\ti n^c}&\T^{s^c\ti r^c}& \Tcp^{s^c\ti k}\\\T^{m\ti n^c}&\Trp^{m\ti r^c}&0}
$$
and a triangular scheduling function
$
\hat{\De}_c:\mathbf{V}\to\T^{r^c\ti s^c}
$; the intersected structures for $\Bh^c_2,\,\Ch^c_2$ naturally follow from $\Kc_3$/$\Kc_4$ with \eqref{clk3}/\eqref{clk4} if viewing \eqref{clk4} in the partitions $k,\,m\in\N_0^2$.
In view of the triangular structure in the control and uncertainty channel, we then choose the variables $X_i,\,Y_i$ as for the triangular case in Sec.~\ref{syn2}, while replacing \eqref{klmn} with
$$
\mat{ccc}{\bs{K_{11}}&\bs{K_{12}}&\bs{L_1}\\\bs{K_{21}}&\bs{K_{22}}&\bs{L_2}\\\bs{M_1}&\bs{M_2}&\bs{N}}\in\mat{lll}{\T^{a\ti u}&\T^{a\ti v}&\T^{a\ti k}\\\T^{b\ti u}&\T^{b\ti v}& \Tcp^{b\ti k}\\\T^{m\ti u}&\Trp^{m\ti v}&0}.
$$
The corresponding synthesis result can then be formulated as in Theorem~\ref{theo:synthesis4}.

Furthermore, our synthesis results for partial and triangular gain-scheduling based on the nested configuration in Fig.~\ref{fig0} can be easily generalized to a hierarchical structure consisting of a finite number of an arbitrary number of subsystems. This requires to adapt the projection terms (\ref{klmn22}) and (\ref{klmn2}), respectively, according to the parameter transformation in \cite{roesinger2019a}.
Finally, the lifting procedure and our synthesis results can be directly applied to other performance specifications, such as the generalized $\Hz$-cost based on the analysis result in \cite{scherer2000}.

\subsection{On the relation of original with lifted gain-scheduling}\label{secopt}

The lifting approach permits us to solve structured gain-scheduling synthesis problems beyond
single-objective $\Hi$-design \cite{scherer2000}. For the latter, however, the question arises
whether lifting introduces any conservatism.
For all classes $\Gc_i$/$\Kc_i$, $i=1,\ldots,4$, this is answered in the negative in this section.

\begin{thm}
Let $\ga_{i,\rm\opt}$ denote the optimal $\Hi$-bound that is achievable
for $\Gc_i$/$\Kc_i$ with the class of multipliers in $\mathbf{\hat{P}}$ which satisfy (\ref{ek1}). Similarly,
let $\ga^l_{i,\rm opt}$ be such an optimal $\Hi$-bound for $\Gc_i^l$/$\Kc_i$ with the multiplier class
$\mathbf{P}$. Then $\ga^l_{i,\rm opt}\leq \ga_{i,\rm\opt}$.
\end{thm}

\begin{proof}
It suffices to give the proof for $\Gc_3$/$\Kc_3$. Choose any $\ga>\ga_{i,\rm\opt}$.
Then there exist $\Xc_1\cg 0$, $\hat{\Pc}\in\mathbf{\hat{P}}$ with (\ref{ek1})
and a controller $K(\De)\in\Kc_3$ which achieve \eqref{AHi} for the interconnection with the original plant $G(\De)\in\Gc_3$.
By the Lifting Lemma~\ref{theo:untolift} and if recalling the definition of $\mathbf{P}$ in \eqref{lscalings}, there exists a multiplier $\Pc\in\mathbf{P}$ such that \eqref{LHi} holds for the same $\ga$ and for the lifted closed loop system \eqref{eq5}-\eqref{eq3lifted}.

It is now crucial to observe that the lifted structured controller \eqref{e2lift} can be again considered as an element of $\Kc_3$. Indeed, if recalling the structures $2\Ah^c_{22}=\smat{\bu&0\\\bu&\bu}$, $2\hat{\De}_c(\De)=\smat{\bu&0\\\bu&\bu}$, and $r^c=(r^c_1,r^c_2)$, $s^c=(s^c_1,s^c_2)$, the matrices $\smat{I_{r^c}&0\\2\Ah^c_{22}&-I_{s^c}}$ and $\smat{-I_{r^c}&2\hat{\De}_c(\De)\\0&I_{s^c}}$ read as
$$
\mat{cc;{2pt/1pt}cc}{I_{r^c_1}&0&0&0\\\bu&-I_{s^c_1}&0&0\\\hdl 0&0&I_{r^c_2}&0\\\bu&0&\bu&-I_{s^c_2}}
\quad\text{and}\quad
\mat{cc;{2pt/1pt}cc}{-I_{r^c_1}&\bu&0&0\\0&I_{s^c_1}&0&0\\\hdl 0&\bu&-I_{r^c_2}&\bu\\0&0&0&I_{s^c_2}}
$$
after swapping the second and third block row/column. This illustrates that a simple permutation of the signals in the scheduling channel of \eqref{e2lift} leads to a description of the lifted controller
with a $2\ti 2$ lower block triangular structure as required for elements in $\Kc_3$ according to \eqref{clk3}.

Therefore, the performance bound $\ga$ can indeed be achieved for some $\Pc\in\mathbf{P}$ and
some $K(\De)\in\Kc_3$ interconnected with the lifted plant \eqref{eq5}, which shows $\ga^l_{i,\rm opt}\leq \ga$. Since $\ga$ with $\ga>\ga_{i,\rm\opt}$ was arbitrary, the proof is finished.
\end{proof}

This result for the class $\Gc_1$/$\Kc_1$ and the $\Hi$-cost shows that the lifting approach is not more conservative than the one proposed in \cite{scherer2000}. Moreover, even if it was possible to convexify gain-scheduled synthesis for the classes $\Gc_2$/$\Kc_2$, $\Gc_3$/$\Kc_3$ or the $\Hz$-cost directly, we can still conclude that lifting would not lead to more conservative results.

\section{Numerical example}\label{secnum}

Let us consider the nested configuration in Fig.~\ref{fig0} for triangular gain-scheduling, i.e., also the outer loop is scheduled as $P_2=P_2(\De)$, $C_2=C_2(\De)$. If $\Delta\in[-1,1]$, let $P_1(\De)$ and $P_2(\De)$ be given by
$$
\begin{aligned}
\smat{\hat{z}_1\\[0.5ex]\xi\\[0.5ex]y_1}&=\tfrac{1}{s-2}\smat{0.9s-0.8\,&s-1\\[0.5ex]1&1\\[0.5ex]s&2}\smat{\hat{w}_1\\[0.5ex]u_1},\
\hat{w}_1=\De\hat{z}_1 \quad \te{and}\\
\smat{\hat{z}_2\\[0.5ex]y_2}&=\tfrac{1}{s-1}\smat{-0.9s+3.9\,&0.6s+2.4&0\\[0.5ex]-s+4&3&s-1}\smat{\hat{w}_2\\[0.5ex]u_2\\[0.5ex]\xi},\
\hat{w}_2=\De\hat{z}_2,
\end{aligned}
$$
respectively. By proceeding as in Sec.~\ref{classes1} for partial gain-scheduling, this leads to the system description
$$
\renewcommand{\arraystretch}{1.1}
\mat{c}{\dot{x}\\\hline\hat{z}\\\hdl y}\,{=}\,\mat{c|c;{2pt/1pt}c}{\Ah_{11}&\Ah_{12}&\Bh_1\\\hline \Ah_{21}&\Ah_{22}&\Bh_2\\\hdl \check C_1&\check C_2&\check D}\mat{c}{x\\\hline \hat{w}\\\hdl u}\\
\,{:=}\,\renewcommand{\arraystretch}{1.0}
	\scalebox{0.77}{$\mat{rr|rr;{2pt/1pt}rr}{2&0&1&0&1&0\\0&1&0&1&0&1\\\hline 1&0&0.9&0&1&0\\0&3&0&-0.9&0&0.6\\\hdl 2&0&1&0&0&0\\1&3&0&-1&0&0}$}
	\renewcommand{\arraystretch}{1.1}
\mat{c}{x\\\hline \hat{w}\\\hdl u}
$$
with $\hat{w}=\hat\Delta(\Delta)\hat{z}$ for $\hat\Delta(\Delta):=\smat{\Delta&0\\0&\Delta}$ which belongs to the class $\Gc_3$ with $n=r=s=m=k=(1,1)$, $m^p=k^p=0$.

Based on Matlab's Robust Control Toolbox, let us perform a standard S/KS-design, i.e., an $\Hi$-design for bounding the weighted sensitivity and control sensitivity transfer matrices $S$ and $KS$, respectively. To this end, we consider the plant
$$
\mat{c}{\dot{x}\\\hline \hat{z}\\\hdl e\\u\\\hdl y}=\matl({c|c;{2pt/1pt}c;{2pt/1pt}c}{
	\Ah_{11}&\Ah_{12}&0&\Bh_1\\\hline
	\Ah_{21}&\Ah_{22}&0&\Bh_2\\\hdl
	-\check C_1&-\check C_2&I&-\check D\\
	0&0&0&I\\\hdl
	-\check C_1&-\check C_2&I&-\check D})\mat{c}{x\\\hline\hat{w}\\\hdl w_p\\\hdl u},\quad
\hat{w}=\mat{cc}{\De&0\\0&\De}\hat{z}
$$
with the performance output $z_p:=\col(e,u)$ which collects the tracking error $e$ and the control input $u$. The corresponding weights are taken as
$w_e(s):=\frac{s+3}{2s + 9\cdot 10^{-4}}$ and $w_u(s):=\frac{s+2}{10^{-2}s+10}\smat{1&2}$, respectively.
The resulting weighted plant is the LFR $G(\De)\in\Gc_3$ used for synthesis.

The goal is the construction of a triangular gain-scheduled controller $K_t(\De)\in\Kc_3$ solving Problem~\ref{problem2}. This is compared to an unstructured gain-scheduled controller $K_u(\De)\in\Kc_1$
which results from viewing $G(\De)$ as an unstructured plant in $\Gc_1$. We compute the optimal $L_2$-gain bound $\ga_\opt$ based on the LMIs in Theorem~\ref{theo:synthesis4},
whose implementation is facilitated by using \cite{lofberg2004} in conjunction with our description of the structured variables and the constraints. The corresponding controller is determined for the increased bound $1.05\,\ga_\opt$ to improve the numerical conditioning of the controller construction.

By using the solver MOSEK, the design for $K_u(\De)$ guarantees an $L_2$-gain bound of $\ga_\opt=6.58$, while that for $K_t(\De)$ leads to the higher value of $\ga_\opt = 11.11$. The difference is the prize to be paid for designing a triangular controller.

The resulting closed-loop frequency responses of $w_p\to e$ are shown in Fig.~\ref{fe} for the triangular (full) and the unstructured (dashed) design, if $\Delta$ takes the values $-1$ (purple), $0$ (blue) and $1$ (red), respectively. The inverse weight $w_e$ (scaled with $\ga_\opt = 11.11$) is shown as a black dashed-dotted curve, while the responses of $w_p\to u$ are omitted for brevity.
The associated controller magnitude responses for $K_t(\De)$ (full) and $K_u(\De)$ (dashed) are given in Fig.~\ref{fk} and exhibit the triangular controller structure of $K_t(\De)$.

The highly varying characteristics of the underlying open-loop plant (not shown to save space) is reflected by the fact that the controller gains in the (2,2)/(1,1)-components change from higher/lower to lower/higher values (for lower frequencies) when moving the parameter from $\De=-1$ (purple) to $\De=1$ (red). This induces the opposite behavior for the sensitivity plots in Fig.~\ref{fe}. Apart from and despite the fact that the structured design achieves an exact decoupling of outer loop references from the inner loop, both designs exhibit similar characteristics in terms of their sensitivity magnitude plots.

In Fig.~\ref{ftrack}, we compare the time-domain tracking behavior for $K_t(\De)$ with that for $K_u(\De)$ with $\De(t):=\sin(\pi t/100)$ (black, dashed) and $w_p=\col(0,r_s)$ for $r_s(t)\,{:=}\,0.5\,\texttt{square}(2\pi t/25)$ (black, full)
generated by the Matlab command \texttt{square}. Note that the reference only acts on $P_2(\De)$ in the outer loop (see  Fig.~\ref{fig0}) and Fig.~\ref{ftrack} depicts the two outputs $y_1$ (upper) and $y_2$ (lower) for $K_t(\De)$ (blue) and for $K_u(\De)$ (red), respectively.

Both designs lead to a similar tracking behavior, while the exact suppression of the outer loop reference signal in the inner loop is violated for $K_u(\De)$ (upper, red). 
These plots once again reflect the
change of the characteristics of $P_2(\Delta)$ from an unstable minimum-phase ($\Delta=-1$) to a non-minimum-phase
($\Delta=1$) system, in accordance with the different peaking behavior of the (2,2) block of the corresponding sensitivity transfer function in Fig.~\ref{fe}.

Let us finally compare these designs with those for
partial gain-scheduling where $P_2$ is taken as $P_2(0)$
and $C_2$ is a $\De$-independent controller in Fig.~\ref{fig0}. 
The resulting frequency response plots in Fig.~\ref{fe} (dotted) exhibit a closed-loop system
which match those for $K_t(\De)$ in the (1,1) block for the parameters $\Delta\in\{-1,0,1\}$,
while they resemble those for $K_t(0)$ (blue) in the (2,2) block.

\begin{rema}
The lifting approach goes through if replacing $\hat{w}=-\hat{w}+2\hat{\De}(\De)\hat{z}$ with $\hat{w}=-\mu\hat{w}+(\mu+1)\hat{\De}(\De)\hat{z}$ for some parameter $\mu>0$ and adapting the lifted block\,/\,LFR in \eqref{liftedblock}/\eqref{eq5} accordingly. However, this 
modification causes no improvements in our numerical example.
\end{rema}

\begin{figure}
	\begin{center}
		\includegraphics[trim=62 0 66 24, clip, height=0.28\textwidth]{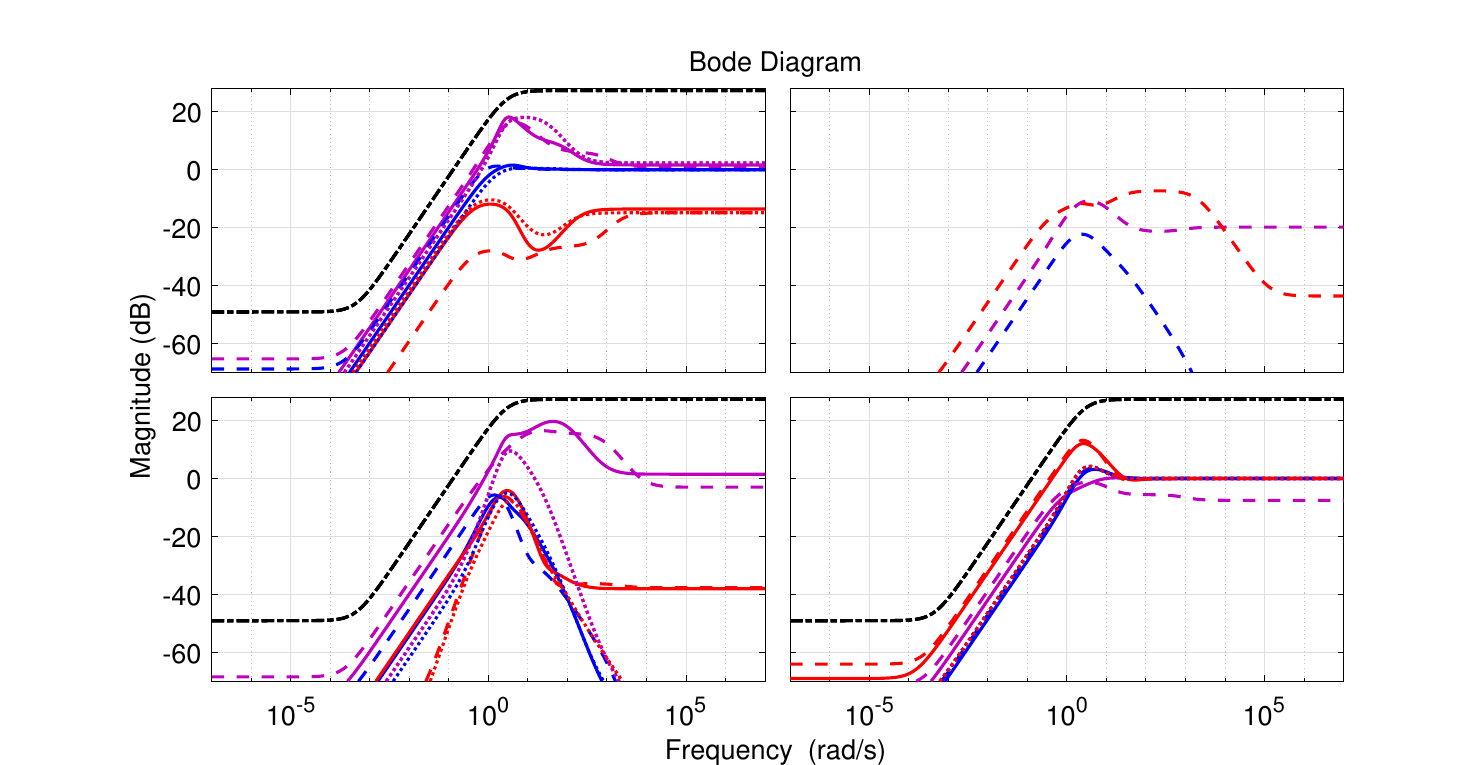}
		\caption{Closed-loop frequency magnitude responses of $w_p\to e$ for triangular (full), unstructured (dashed), partial (dotted) gain-scheduling computed for $\De\,{=}\,-1$ (purple), $\De\,{=}\,0$ (blue) and $\De\,{=}\,1$ (red), and the scaled inverse weight for the triangular design (black, dashed-dotted).}
		\label{fe}
	\end{center}
\end{figure}
\begin{figure}
	\begin{center}
		\includegraphics[trim=62 0 67 24, clip, height=0.28\textwidth]{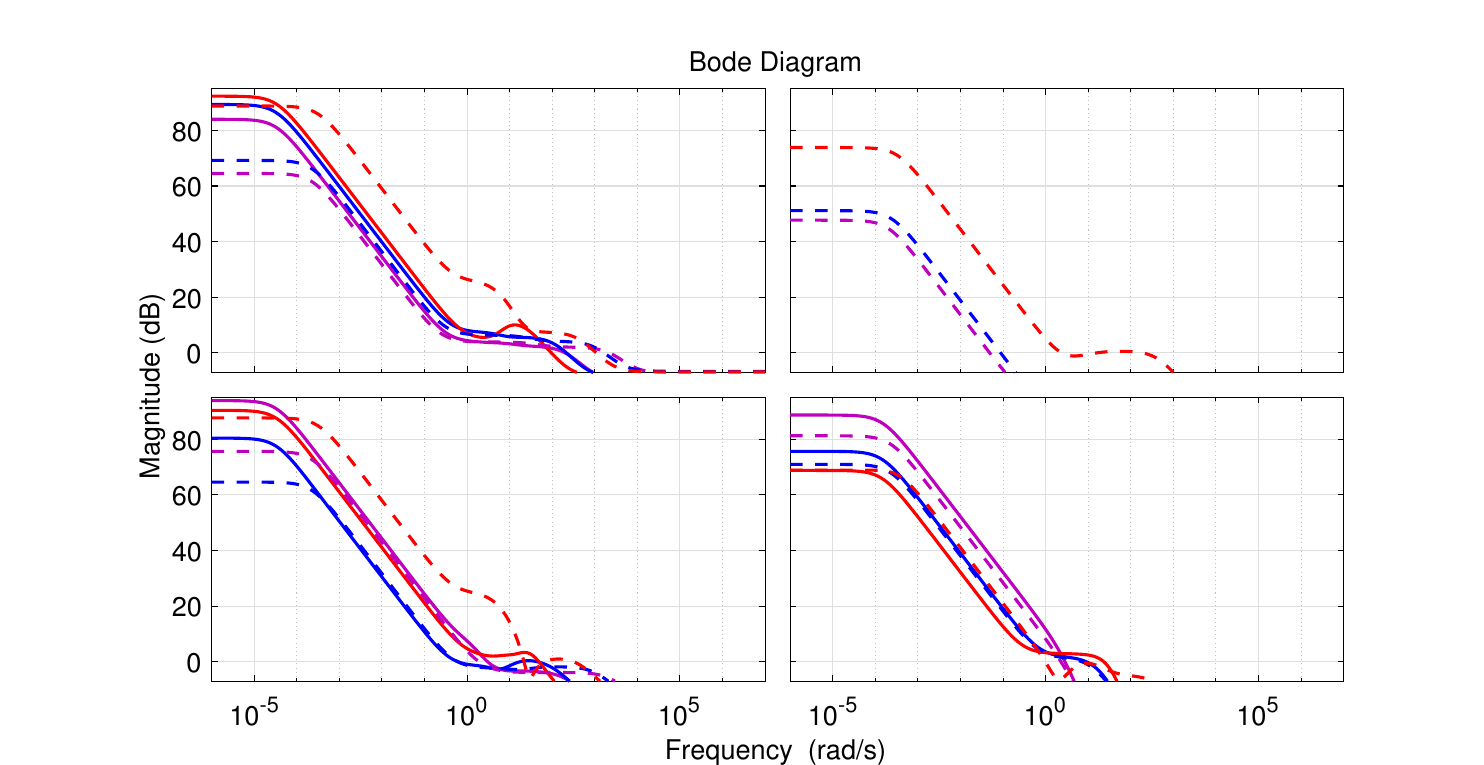}
		\caption{Frequency magnitude responses of the triangular gain-scheduled controller $K_t(\De)$ (full) and the unstructured gain-scheduled controller $K_u(\De)$ (dashed) computed for $\De=-1$ (purple), $\De=0$ (blue) and $\De=1$ (red).}
		\label{fk}
	\end{center}
\end{figure}

\section{Conclusions}
We have presented a lifting technique that allows us to develop a direct synthesis framework for nested gain-scheduling with multiple objectives using full block scalings. Based on a novel structured factorization of these scalings, this approach is shown to be successful for solving the partial and triangular gain-scheduling problem related to a nested inner and outer loop configuration. A future goal is the extension to more complicated hierarchical structures and the investigation of numerical advantages of the lifting technique over existing LMI controller design approaches.

\appendix
\section*{Proof of Theorem~\ref{theo:synthesis4}}
	\textit{Necessity.}
	Suppose that \eqref{LHi} in Theorem~\ref{theo:analysisl} holds with some $\Xc_1\cg0$, some scaling $\Xc_2:=\Pc\in\mathbf{P}$, and for the closed-loop system \eqref{eq3lifted}, (\ref{delc}) obtained for the lifted LFR $G_l(\De)\in\Gc_3^l$ and some controller $K(\De)\in\Kc_3$ with $l^c=r^c=s^c$. We infer
	\begin{equation}\label{Lfactor}
	\Lc\left(\smat{0&I\\I&0},\smat{0&I\\I&0},\smat{-\ga I&0\\0&I};\smat{\Xc_1\Ac_{1j}&\Xc_1\Bc_1\\\Xc_2\Ac_{2j}&\Xc_2\Bc_2\\\Cc_j&\Dc}\right)\cl0.
	\end{equation}
	In the first steps, we properly factorize $\Xc_i$ according to \eqref{fac}
	where the structures of $X_i,\,Y_i,\,U_i,\,V_i$ depend on $\Gc_3,\,\Kc_3$ and on the structure of the lifted LFR.

	\textbf{Step 1} (\textit{Factorization of Lyapunov matrix $\Xc_1$}).\\
	In $\Gc_3$/$\Kc_3$, the plant/controller matrices related to the control channel are triangular, which motivates to use the factorization in \cite{scherer2013}, \cite{scherer2014} for convexification.
	By assuming w.l.o.g. that $n^c_1\geq |n|$ and $n^c_2\geq |n|$, it is shown in these papers that there exists a factorization (\ref{fac}) for $i=1$ such that $\Yc_1$ has full column rank and such that the factors admit the structure
	\begin{equation}\label{facl1}
	\mat{c}{Y_1\\\hdl V_1}=\mat{cc}{
		\bs{Y_1}&\bs{Y_2}\\\hdl \bu&0\\ \bu&\bu},\quad
	\mat{c}{X_1\\\hdl U_1}=\mat{cc}{\bs{X_2}&\bs{X_3}\\\hdl \bu&\bu\\0&\bu}
	\end{equation}
	with variables defined as in \eqref{tgs1}-\eqref{ev1}, and with triangular matrices $V_1\in\T^{n^c\ti u}$, $U_1^T\in\T^{a\ti n^c}$ with $u,\,a$ in (\ref{partitionex}).
	As in Step~1 of the proof of Theorem~\ref{theo:synthesis1}, we get $\Zc_1^T\Yc_1\cg0$. By symmetry, the latter reads as $\Xhb\cg0$ with (\ref{couplingt}).

\begin{figure}
	\begin{center}
		\includegraphics[trim=45 2 40 10, clip, width=0.497\textwidth,height=0.264\textwidth]{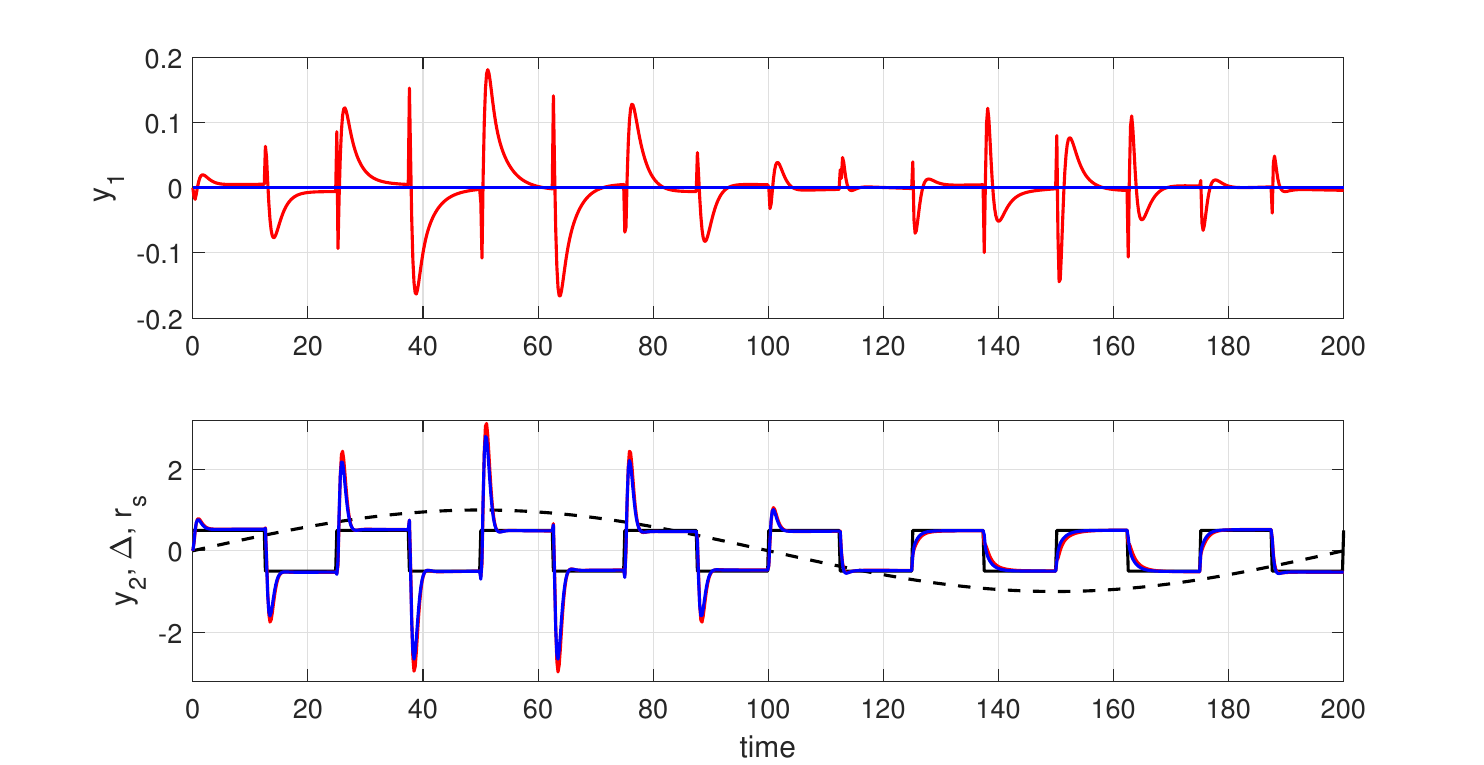}
		\caption{
		Responses $y_1$ (upper) and $y_2$ (lower) for gain-scheduling with $K_t(\De)$ (blue) and $K_u(\De)$ (red). The lower plot contains the scheduling signal $\De$ (black, dashed) and the square reference $r_s$ (black, full).
		}
		\label{ftrack}
	\end{center}
\end{figure}

	\textbf{Step 2} (\textit{Factorization of scaling matrix $\Xc_2$}).\\
	For $\Xc_2$, we construct a novel factorization which is shown to be the right choice for $\Gc_3$/$\Kc_3$.
First, if $\varepsilon>0$, note that the LFR of $K(\De)$ with its specific structure in \eqref{clk3} does not change if increasing the size of the scheduling channel by diagonally extending  $\Ah^c_{22}$ with $-\varepsilon I$, $\hat{\De}_c(\De)$ with $\varepsilon I$, and $\Xc_2$ with $\varepsilon I$, while properly increasing $\Ah^c_{12},\,\Ah^c_{21},\,\Bh^c_2,\,\Ch^c_2$ by zero blocks.
Therefore, we can assume w.l.o.g. that $l^c_1\geq |l|$ and $l^c_2\geq |l|$
for the components of $l^c=(l^c_1,l^c_2)$ defining the partition of the controller's scheduling block,
while \eqref{LHi} and $\Xc_2\in\mathbf{P}$ remain to be true. Then we can partition $\Xc_2$ according to $(|l|,(l^c_1,l^c_2))=(|l|,l^c)$ as
	\begin{equation}\label{X2}
	\Xc_2
	=\mat{c;{2pt/1pt}cc}{
		\bs{P_3}&S_{13}^T&S_{23}^T\\\hdl
		S_{13}&Z_{11}&Z_{12}\\
		S_{23}&Z_{21}&Z_{22}}=:\mat{c;{2pt/1pt}c}{\bs{P_3}&S_1^T\\\hdl S_1&Z_1}.
	\end{equation}
We first target at the factorization
	\begin{equation}\label{fact}
	\Xc_2\mat{cc;{2pt/1pt}c}{T_1&T_2&I_{l}\\\hdl
		T_{11}&0&0\\T_{21}&T_{22}&0}
	=\mat{c;{2pt/1pt}cc}{I_{l}&I_{l}&\bs{P_3}\\\hdl
		0&\t{S}_{12}&\t{S}_{13}\\0&0&\t{S}_{23}}.
	\end{equation}
	Since \eqref{LHi} and the scaling inequality $\He\left[\Xc_2\De_{lc}(\De)\right]\cg\nolinebreak0$ for $\De\in\mathbf{V}$ are strict and since $\mathbf{V}$ is compact, we can successively infer invertibility of $Z_{22},\,Z_1$ by
perturbing $Z_{22},\,Z_{11}$ such that \eqref{LHi} and $\Xc_2\in\mathbf{P}$ stay valid.\\
Further, by partitioning $\t{Z}_1:=Z_1^{-1}$ as $Z_1$ in \eqref{X2}, the block-inversion formula shows invertibility of $\t{Z}_{jj}$ for $j=1,2$.
	In particular, since $S_{j3}$ is tall due to $l^c_j\geq|l|$ and since $\t{Z}_{jj}$ is invertible,
we can successively perturb $S_{23},\,S_1$ such that
	\begin{equation}\label{rank}
	H_1:=-\mat{cc}{I_{l^c_1}&0}\t{Z}_1S_1
	\quad\text{and}\quad
	H_2:=-\t{Z}_{22}S_{23}
	\end{equation}
    have full column rank and \eqref{LHi}, $\Xc_2\in\mathbf{P}$ stay true.
Moreover, perturbing $\bs{P_3}$ allows to infer invertibility of
	\begin{equation}\label{perturbation}
	\bs{P_3}-S_1^T\t{Z}_1S_1\quad \text{and}\quad \bs{P_3}-S_{23}^T\t{Z}_{22}S_{23}
	\end{equation}
	without violating \eqref{LHi} and $\Xc_2\in\mathbf{P}$.
E.g., with the original $\bs{P_3}$
we can choose $\eps\in\R$ arbitrarily close to zero such that
\begin{equation}\label{perturbation2}
-\eps\notin\eig(\bs{P_3}-S_1^T\t{Z}_1S_1)
\cup\eig(\bs{P_3}-S_{23}^T\t{Z}_{22}S_{23})
\end{equation}
and then replace $\bs{P_3}$ with $\bs{P_3}+\varepsilon I$.\\
  Invertibility of (\ref{perturbation}) and of $Z_1,\,Z_{22}$ implies the same for $\Xc_2$ and $\smat{\bs{P_3}&S_{23}^T\\ S_{23}&Z_{22}}$, which allows us to define
	\begin{equation}\label{Yj}
	\mat{c}{T_1\\\hdl T_{11}\\T_{21}}:=\Xc_2^{-1}\mat{c}{I_l\\\hdl 0\\ 0},\ \
	\mat{c}{T_2\\\hdl T_{22}}
	:=\mat{c;{2pt/1pt}c}{\bs{P_3}&S_{23}^T\\\hdl S_{23}&Z_{22}}^{-1}
	\mat{c}{I_l\\\hdl 0}.
	\end{equation}
	This leads to \eqref{fact} with suitable blocks $\t{S}_{ij}$.

Next we show that there exists $\bs{P_2},\,\bs{Q_2}$ as in \eqref{ev3} satisfying
\begin{equation}\label{ev4}
	T_2\bs{P_2}=\bs{Q_2}.
	\end{equation}
	For this purpose, let $\Pi\in\R^{l\ti l}$ be the permutation matrix
for which the second and third block row of
	$$
	\Pi T_2 \quad\te{and}\quad \Pi\left[\bs{P_3}-S_{23}^T\t{Z}_{22}S_{23}\right]
	$$ are swapped in the partition $l=(r_1,r_2,s_1,s_2)$, and define
	\begin{equation}\label{defnsub}
	\smat{\check{Q}_1&\check{Q}_2^T\\\check{Q}_2&\check{Q}_3}:=\Pi T_2\Pi^T,\
	\smat{\check{P}_1&\check{P}_2^T\\\check{P}_2&\check{P}_3}:=\Pi\left[\bs{P_3}-S_{23}^T\t{Z}_{22}S_{23}\right]\Pi^T
	\end{equation}
	of dimension $((r_1,s_1),(r_2,s_2))$. Let us observe that by adjoining $\eig(\check{P}_3)$ to the right-hand side of \eqref{perturbation2}, \eqref{defnsub} allows to modify the above perturbation of $\bs{P_3}$ in order to achieve invertibility of \eqref{perturbation} and of $\check{P}_3$.
	Further, due to \eqref{Yj}, the block-inversion formula shows $T_2=\bigl[\bs{P_3}-S_{23}^T\t{Z}_{22}S_{23}\bigr]^{-1}$ and thus $\smat{\check{Q}_1&\check{Q}_2^T\\\check{Q}_2&\check{Q}_3}=\smat{\check{P}_1&\check{P}_2^T\\\check{P}_2&\check{P}_3}^{-1}$. Hence, $\check{Q}_1=\bigl[\check{P_1}-\check{P_2}^T\check{P_3}^{-1}\check{P_2}\bigr]^{-1}$ is invertible by construction.
By inspection we get
	\begin{equation}\label{ev5}
	\Pi T_2\Pi^T \smat{\check{Q}_1^{-1}&-\check{Q}_1^{-1}\check{Q}_2^T\\0&I}=\smat{I&0\\\check{Q}_2\check{Q}_1^{-1}&\check{Q}_3-\check{Q}_2\check{Q}_1^{-1}\check{Q}_2^T}
	\end{equation}
	and we can choose
	$$
	\begin{aligned}
	\smat{\bs{P_{22}}&\bs{P_{23}}\\\bs{P_{32}}&\bs{P_{33}}}&:=\check{Q}_1^{-1}
	,\quad
	\smat{\bs{Q_{22}}&\bs{Q_{23}}\\\bs{Q_{32}}&\bs{Q_{33}}}:=\check{Q}_3-\check{Q}_2\check{Q}_1^{-1}\check{Q}_2^T
	,\\
	\smat{\bs{R_{22}}&\bs{R_{23}}\\\bs{R_{32}}&\bs{R_{33}}}&:=-\check{Q}_2\check{Q}_1^{-1}
	\end{aligned}
	$$
as in (\ref{ev2}) to define $\bs{P_2},\,\bs{Q_2}$ through (\ref{ev3}). Clearly, (\ref{ev5}) transforms into (\ref{ev4}) after left/right-multiplication with $\Pi^T/\Pi$.

Setting $\bs{Q_1}:=T_1$ and right-multiplying the second block column of (\ref{fact}) with $\bs{P_2}$ while using (\ref{ev4}) leads to  the factorization
of $\Xc_2$ as in (\ref{fac}) for $i=2$ and with
\begin{equation}\label{fac31}
	\mat{c}{Y_2\\\hdl V_2}=
	\mat{cc}{
		\bs{Q_1}&\bs{Q_2}\\\hdl \bu&0\\ \bu&\bu},\ \
	\mat{c}{X_2\\\hdl U_2}=\mat{cc}{\bs{P_2}&\bs{P_3}\\\hdl \bu&\bu\\0&\bu}.
	\end{equation}
By construction, note that (\ref{tgs2})-(\ref{ev2}) hold and that the partitions of
$V_2\in\T^{l^c\ti v},\, U_2^T\in\T^{b\ti l^c}$ can be expressed with $v,b$ in (\ref{partitionex}).

Let us convince ourselves that $\Yc_2$ has full column rank.
	By construction, $H_j$ in (\ref{rank}) has full column rank. By the block-inversion formula, (\ref{Yj}) implies that $T_j$ is invertible and, hence, $T_{jj}=H_jT_j$ has full column rank as well. Therefore, this holds for the right-factor of $\Xc_2$ in (\ref{fact}). Since $\bs{P_2}$ is as well invertible, we conclude that $\Yc_2$ has full column rank.

	\textbf{Step 3} (\textit{Elimination of scheduling function $\hat{\De}_c$}).\\
	As in Step~3 of the proof of Theorem~\ref{theo:synthesis1}, we omit the arguments of $\De_l(.)$, $\hat{\De}_c(.)$, $\De_{lc}(.)$ and infer (\ref{exdec}) from $\Xc_2\in\mathbf{P}$.
	With the partitions of $X_2,\,Y_2$ in (\ref{fac31}), let us note that
	\begin{equation}\label{subblocks}
	\He\left[\mat{c}{I\\\hdl X_2^T}\De_l\mat{c;{2pt/1pt}c}{Y_2&I}\right]
	=:\He\mat{cc;{2pt/1pt}c}{\De_l\bs{Q_1}&0&0\\\hdl
			J_{21}&\bs{P}^T_{\hspace{-0.5ex}\bs{2}}\De_l\bs{Q_2}&0\\
			J_{31}&J_{32}&\bs{P}^T_{\hspace{-0.5ex}\bs{3}}\De_l}
	\end{equation}
	holds with blocks $J_{ij}=J_{ij}(\De_l,\De_l^T)$ that depend on $\De_l$ and $\De_l^T$, respectively. Moreover, since $U_2,\,V_2,\,\hat{\De}_c$ are triangular, we infer $U_2^T\hat{\De}_cV_2\in\T^{b\ti v}$. This implies that the diagonal blocks of
		\eqref{exdec} are exactly given by those of \eqref{subblocks}, and hence
	\begin{equation}\label{diagonal}
	\He\left[\De_l\bs{Q_1}\right]\cg0,\ \He\left[\bs{P}^T_{\hspace{-0.5ex}\bs{2}}\De_l\bs{Q_2}\right]\cg0,\
	\He\left[\bs{P}^T_{\hspace{-0.5ex}\bs{3}}\De_l\right]\cg0.
	\end{equation}
	Therefore, we conclude \eqref{couplings}, while $\bs{Q_1}\in \mathbf{P}_d$ and $\bs{P_3}\in\mathbf{P}_p$ follow from \eqref{diagonal} as shown in Step~3 of the proof of Theorem~\ref{theo:synthesis1}.

	\textbf{Step 4} (\textit{Convexifying parameter transformation}).\\
	Since $\Yc_i$ has full column rank, let us use congruence transformations with $\Yc_i/\Yc_i^T$ on the $i$-th block-column/-row of \eqref{Lfactor} together with the factorization (\ref{fac}) to infer
	$$
	\Lc\left(\smat{0&I\\I&0},\smat{0&I\\I&0},\smat{-\ga I&0\\0&I};\smat{\Zc_1^T\Ac_{1j}\Yc_j&\Zc_1^T\Bc_1\\\Zc_2^T\Ac_{2j}\Yc_j&\Zc_2^T\Bc_2\\\Cc_j\Yc_j&\Dc}\right)\cl0.
	$$
	As in Step~4 of the proof of Theorem~\ref{theo:synthesis1}, we can express the matrices in the outer factor as (\ref{s19})
	after performing the substitution (\ref{s2}). We observe that $N=\hat{D}^c=:\bs{N}$ already has the desired structure (\ref{klmn}). Moreover, let us define
	\begin{equation}\label{s3}
	\bs{K_{ij}}:=\Lo(K_{ij}),\ \bs{L_i}:=\Lo(L_i),\ \bs{M_j}:=\Lo(M_j)
	\end{equation}
	to trivially infer the structure (\ref{klmn}).
	
	In order to finish the necessity proof, it remains to show \eqref{klmn2}. In view of $K(\De)\in\Kc_3$ with (\ref{clk3}), all controller matrices are lower block-triangular. The same is true for $U_i^T,\,V_j$ by construction. This implies
	\begin{equation}\label{s4}
	\U(U_i^T\hat{A}_{ij}^cV_j)=0,\ \U(U_i^T\hat{B}^c_i)=0,\
	\U(\hat{C}^c_jV_j)=0.
	\end{equation}
	Hence, by exploiting (\ref{s3}), (\ref{s4}), and linearity of $\U(.)$, we can express (\ref{s2}) as
	\begin{equation}\label{s5}
	\begin{aligned}
	L_i&=\bs{L_i}+\U(L_i)=\bs{L_i}+\U(X_i^TB_i\bs{N}),\\
	M_j&=\bs{M_j}+\U(M_j)=\bs{M_j}+\U(\bs{N}C_jY_j),
	\end{aligned}
	\end{equation}
	and $K_{ij}=\bs{K_{ij}}+\U(K_{ij})$. This yields
	\begin{equation}\label{s6}
	\begin{aligned}
	K_{ij}&=\bs{K_{ij}}+\U(X_i^TA_{ij}Y_j)+\U(X_i^TB_iM_j)+\\
	&\phantom{=}+\U(L_iC_jY_j)-\U(X_i^TB_i\bs{N}C_jY_j).
	\end{aligned}
	\end{equation}

	Let us next show that the non-linear expressions (\ref{s5}), (\ref{s6}) can indeed be turned into (\ref{klmn2}) for the new decision variables.
	As a consequence of the particular structure of \eqref{ev1}, \eqref{a1}, we infer
	\begin{equation}\label{r1}
	\bs{X}^T_{\bs{2}}\t B_1=\t B_1\quad\text{and}\quad \t C_1\bs{Y_2}=\t C_1
	\end{equation}
	by a direct computation. At this point it is crucial to see, with the structure of the lifted LFR (\ref{eq5}), that we can introduce the partitions $B_2\in\T^{(|r|+s_1,s_2)\ti m}$ and $C_2\in\T^{k\ti (r_1,r_2+|s|)}$. This directly shows
	$$
	\t B_2=\smat{0\\\bu}\in\R^{(|r|+s_1,s_2)\ti m_2}\ \text{and}\ \t C_2=\smat{\bu&0}\in\R^{k_1\ti(r_1,r_2+|s|)}
	$$
	in \eqref{a1}. By exploiting the structure of $\bs{P_2},\,\bs{Q_2}$ in \eqref{ev3}, this implies the following relations in analogy to \eqref{r1}:
	\begin{equation}\label{r1new}
	\bs{P}^T_{\bs{2}}\t B_2=\t B_2\quad\text{and}\quad \t C_2\bs{Q_2}=\t C_2.
	\end{equation}
	Now, due to $B_i^e,\,C_j^e$ in \eqref{a2}, we infer from \eqref{r1} and \eqref{r1new} that
	$$
	\U(X_i^TB_i)=B_i^e\quad\text{and}\quad \U(C_jY_j)=C_j^e
	$$
	by a direct computation. This leads to the key relations
	\begin{equation}\label{r2}
	X_i^TB_i=B_i^e+\Lo(X_i^TB_i)\ \text{and}\ C_jY_j=C_j^e+\Lo(C_jY_j).
	\end{equation}
	Analogously to nominal $p$-block triangular synthesis [\citen{roesinger2019a}, eqs. (22)-(24)], it suffices to apply \eqref{r2} and projection rules for $\U,\,\Lo$ to express the projection terms in \eqref{s5}, \eqref{s6} as
	\begin{equation}\label{r3}
	\begin{aligned}
	\U(X_i^TB_i\bs{N})&=\U(B_i^e\bs{N}),\\
	\U(\bs{N}C_jY_j)&=\U(\bs{N}C_j^e),
	\end{aligned}
	\end{equation}
	and
	\begin{equation}\label{r4}
	\begin{aligned}
	&\U(X_i^TB_iM_j)+\U(L_iC_jY_j)-\U(X_i^TB_i\bs{N}C_jY_j)=\\
	&\hspace{0ex}=\U(B_i^e\bs{M_j})+\U(\bs{L_i}C_j^e)-\U(B_i^e\bs{N}C_j^e).
	\end{aligned}
	\end{equation}
	Due to $\U(B_i^e\bs{N}C_j^e)=B_i^e\bs{N}C_j^e$, we immediately infer \eqref{klmn2}. In order to avoid duplications, let us only display this mechanism for $\U(\bs{N}C_jY_j)$.
	Indeed, by using \eqref{r2} combined with the linearity of $\U(.)$, the simple identity $\bs{N}=\Lo(\bs{N})$ together with $\U\circ\Lo=0$ gives
	$$
	\begin{aligned}
	\U(\bs{N}C_jY_j)&\,{=}\,\U(\bs{N}C_j^e){+}\U(\Lo(\bs{N})\Lo(C_jY_j))\\
	&\,{=}\,\U(\bs{N}C_j^e)
	\end{aligned}
	$$
	which is \eqref{r3}. All this finishes the necessity proof.
			
	\textit{Sufficiency.} Let \eqref{SHi} and \eqref{couplings} be feasible.
	
	\textbf{Step 1} (\textit{Construction of Lyapunov matrix $\Xc_1$, multiplier $\Xc_2$})\\
	As in the sufficiency proof of Theorem~\ref{theo:synthesis1}, choose invertible blocks
	$U_i,\,V_i$ of the matrices $\Yc_i,\,\Zc_i$ in (\ref{fac}) such that $\Yc_i^T\Zc_i=\smat{Y_i^T&Y_i^TX_i+V_i^TU_i\\I&X_i}$ is symmetric for $X_i,\,Y_i$ from \eqref{tgs1}, \eqref{tgs2}. In particular, as motivated by the factorizations (\ref{fac}) with \eqref{facl1}, \eqref{fac31}, we ensure the right triangular structure of $U_i,\,V_i$.
	
	This is achieved for $i=1$ as in the nominal case [\citen{scherer2014}, proof of Theorem~2] by taking the lower triangular matrices
	$$
	U_1^T:=\mat{cc}{I_{n}&0\\0&I_{n}}\ \text{and}\ V_1:=\mat{cc}{\bs{Y}^T_{\bs{2}}-\bs{X_2}^{\hspace{-0.8ex}T}\bs{Y_1}&0\\I-\bs{X_3}^{\hspace{-1.1ex}T}\bs{Y_1}&\bs{X_2}-\bs{X_3}^{\hspace{-1.1ex}T}\bs{Y_2}}
	$$
	which are proven to be invertible in \cite{scherer2014} due to $\Xhb\cg0$.
	
	If $i=2$, we infer symmetry of $\Yc_2^T\Zc_2$ by taking, as for $i=1$, the $2|l|\ti 2|l|$ lower triangular matrices
	$$
	U_2^T:=\mat{cc}{I_{l}&0\\0&I_{l}}\ \text{and}\
	V_2:=\mat{cc}{\bs{Q}^T_{\bs{2}}-\bs{P}^T_{\bs{2}}\bs{Q_1}&0\\
	I-\bs{P}^T_{\hspace{-0.5ex}\bs{3}}\bs{Q_1}&\bs{P_2}-\bs{P}^T_{\hspace{-0.5ex}\bs{3}}\bs{Q_2}}.
	$$
	By perturbation of $X_2,\,Y_2$ and compactness of $\mathbf{V}$, we can infer invertibility of $V_2$ such that the strict inequalities in \eqref{SHi}, \eqref{couplings} and for $\bs{P_3}\in\mathbf{P}_p$, $\bs{Q_1}\in\mathbf{P}_d$ persist  to hold for all $\De\in\mathbf{V}$.
	
	Since $U_i,\,V_i$ are invertible, the same is true for $\Yc_i,\,\Zc_i$ in (\ref{fac}). After congruence transformation with $\Yc_i^{-1}$, we infer symmetry of $\Xc_i:=\Zc_i\Yc_i^{-1}$ and validity of the factorization in \eqref{fac}. In particular, $\Yc_1^T\Zc_1=\Xhb\cg0$ directly implies $\Xc_1\cg0$.
	
	\textbf{Step 2} (\textit{Construction of scheduling function $\hat{\De}_c$}).\\
	We drop again the arguments of $\De_l(.)$, $\hat{\De}_c(.)$, $\De_{lc}(.)$ and recall \eqref{subblocks}
	with entries $J_{ij}=J_{ij}(\De_l,\De_l^T)\in\R^{l\ti l}$ depending on $\De_l$ and $\De_l^T$.
	As in the proof of Theorem~\ref{theo:synthesis1}, $\bs{Q_1}\in \mathbf{P}_d$ and $\bs{P_3}\in\mathbf{P}_p$ imply the first and last inequality in \eqref{diagonal}, while the second one holds due to \eqref{couplings}. Since $U_2,\,V_2$ are invertible and triangular, let us choose
	$
	\hat{\De}_c:=
	-U_2^{-T}\smat{J_{21}&0\\
		J_{31}&J_{32}}V_2^{-1}
	$
	as the scheduling function; by definition of $J_{ij}$, we get the explicit formula
	$$
	\hat{\De}_c\,{=}\,-U_2^{-T}\mat{cc}{\bs{P}^T_{\bs{2}}\De_l\bs{Q_1}\!+\bs{Q}^T_{\bs{2}}\De_l^T&0\\\bs{P}^T_{\bs{3}}\De_l\bs{Q_1}+\De_l^T&\bs{P}^T_{\bs{3}}\De_l\bs{Q_2}+\De_l^T\bs{P_2}}V_2^{-1}\!.
	$$
	This shows affine dependence of $\hat{\De}_c$ on $\De$, $\De^T$ since $\De_l$ is affinely dependent on $\De$. For this choice of $\hat{\De}_c$, (\ref{exdec}) reads as
	$$
	\diag\left(\He\left[\De_l\bs{Q_1}\right],
	\He\left[\bs{P}^T_{\bs{2}}\De_l\bs{Q_2}\right],
	\He\left[\bs{P}^T_{\bs{3}}\De_l\right]\right)\cg0
	$$
    and is indeed satisfied for all $\Delta\in\mathbf{V}$ due to \eqref{diagonal}.
	Hence, a congruence transformation involving $\Yc_2^{-1}$ leads back to (\ref{X2Dlc}) by using (\ref{fac}) for $i=2$. Thus, $\Xc_2\in\mathbf{P}$.
	
\textbf{Step 3} (\textit{Construction of controller matrices}).\\
	Due to the structure of the design variables, we can verify \eqref{r3}, \eqref{r4} as in the necessity part. Hence, \eqref{klmn2} leads to \eqref{s5}, \eqref{s6}. Since $U_i,\,V_i$ are invertible and triangular, we can define
	$$
	\begin{aligned}
	\hat{A}_{ij}^c&:=U_i^{-T}\Lo(K_{ij}-X_i^TA_{ij}Y_j-X_i^TB_iM_j-\\
	&\hspace{20ex}-L_iC_jY_j+X_i^TB_i\bs{N}C_jY_j)V_j^{-1},\\
	\hat{B}_i^c&:=U_i^{-T}\Lo(L_i-X_i^TB_i\bs{N}),\\
	\hat{C}_j^c&:=\Lo(M_j-\bs{N}C_jY_j)V_j^{-1}	\qquad\text{and}\qquad \hat{D}^c:=\bs{N}
	\end{aligned}
	$$
	to infer the desired controller structure of $\Kc_3$ in \eqref{clk3} and \eqref{s4}. Moreover, these choices and \eqref{s5}, \eqref{s6} imply the validity of the relation  \eqref{s2}. Congruence transformations with $\Yc_i^{-1}$ and the factorizations (\ref{fac}) lead back to \eqref{LHi}.
\hfill $\blacksquare$

\begin{IEEEbiography}[{\includegraphics[width=1in,height=1.25in]{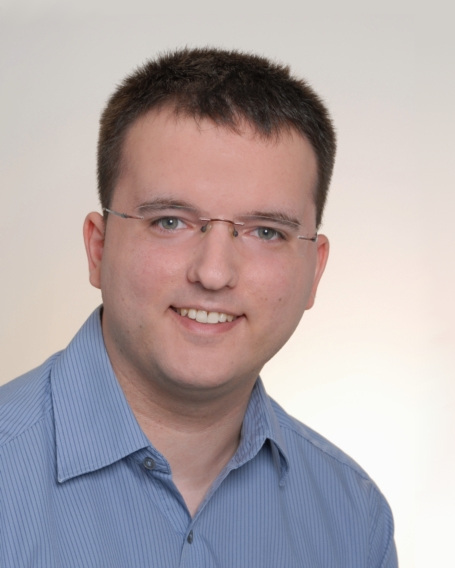}}]{Christian~A.~R\"osinger}
	received the M.Sc. degree in mathematics in 2016 from the University of Stuttgart, Germany, where he is currently pursuing the Ph.D. degree in mathematics.
	
	He is currently affiliated
	with the Chair of Math\nolinebreak ematical Systems Theory at the Department of Mathematics, University of Stuttgart.
\end{IEEEbiography}

\begin{IEEEbiography}[{\includegraphics[width=1in,height=1.25in]{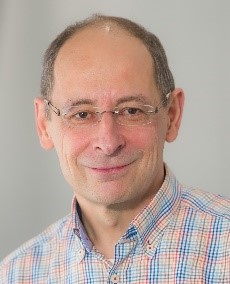}}]{Carsten~W.~Scherer}
	(F'13)
received his Ph.D. degree in mathematics from the University of W\"urzburg (Germany) in 1991. In 1993, he joined Delft University of Technology (The Netherlands) where he held positions as an assistant and associate professor. From December 2001 until February 2010 he was a full professor within the Delft Center for Systems and Control at Delft University of Technology. Since March 2010 he holds the Chair for Mathematical Systems Theory in the Department of Mathematics at the University of Stuttgart (Germany).

Dr. Scherer acted as the chair of the IFAC technical committee on Robust Control (2002-2008), and he has served as an associated editor for the IEEE Transactions on Automatic Control, Automatica, Systems and Control Letters and the European Journal of Control. Since 2013 he is an IEEE fellow ``for contributions to optimization-based robust controller synthesis''.

Dr. Scherer's main research activities cover various directions in developing advanced optimization-based controller design algorithms and their application to mechatronics and aerospace systems.
\end{IEEEbiography}


\begin{thebibliography}{00}

\bibitem{packard1994}
A.~Packard, ``Gain scheduling via linear fractional transformations,'' \emph{Syst. Control Lett.}, vol.~22, no.~2, pp. 79--92, 1994.

\bibitem{apkarian1995}
P.~Apkarian and P.~Gahinet, ``A convex characterization of gain-scheduled $\mathcal{H}_\infty$ controllers,'' \emph{IEEE Trans. Autom. Control}, vol.~40, no.~5, pp. 853--864, 1995.

\bibitem{tien2016}
H.~N.~Tien, C.~W.~Scherer, J.~M.~A.~Scherpen, and V.~M\"uller, ``Linear parameter varying control of doubly fed induction machines,'' \emph{IEEE Trans. Ind. Electron.}, vol.~63, no.~1, pp. 216--224, 2016.

\bibitem{mustaki2019}
S.~Mustaki, A.-T.~Nguyen, P.~Chevrel, M.~Yagoubi and F.~Fauvel, ``Comparison of two robust static output feedback {$H_2$} design approaches for car lateral control,'' in \emph{Proc. 18th Eur. Control Conf.}, 2019, pp. 716--723.

\bibitem{becker1995}
G.~Becker, ``Parameter-dependent control of an under-actuated mechanical system,'' in \emph{Proc. 34th IEEE Conf. Decision and Control}, 1995, pp. 543--548.

\bibitem{wu1996}
F.~Wu, X.~H.~Yang, A.~Packard, and G.~Becker, ``Induced {$L_2$}-norm control for {LPV} systems with bounded parameter variation rates,'' \emph{Int. J. Robust Nonlin.}, vol.~6, no.~9-10, pp. 983--998, 1996.

\bibitem{helmersson98}
A.~Helmersson, ``$\mu$ synthesis and {LFT} gain scheduling with real uncertainties,'' \emph{Int. J. Robust Nonlin.}, vol.~8, no.~7, pp. 631--642, 1998.

\bibitem{scorletti1998}
G.~Scorletti and L.~{El~Ghaoui}, ``Improved {LMI} conditions for gain scheduling and related control problems,'' \emph{Int. J. Robust Nonlin.}, vol.~8, no.~10, pp. 845--877, 1998.

\bibitem{voulgaris2000}
P.~G.~Voulgaris, ``Control of nested systems,'' in \emph{Proc. Amer. Control Conf.}, 2000, pp. 4442--4445.

\bibitem{blanchini2012}
F.~Blanchini, D.~Casagrande, S.~Miani, and U.~Viaro, ``An {LPV} control scheme for induction motors,'' in \emph{Proc. 51st IEEE Conf. Decision and Control}, 2012, pp. 7602--7607.

\bibitem{bucc2019}
L.~Buccafusca, J.~P.~Jansch-Porto, G.~E.~Dullerud, and C.~L.~Beck, ``An application of nested control synthesis for wind farms,'' \emph{IFAC-PapersOnLine}, vol.~52, no.~20, pp. 199--204, 2019.

\bibitem{roesinger2019b}
C.~A.~R\"osinger and C.~W.~Scherer, ``A scalings approach to $\mathcal{H}_2$-gain-scheduling synthesis without elimination,'' \emph{IFAC-PapersOnLine}, vol.~52, no.~28, pp. 50--57, 2019.

\bibitem{roesinger2020}
C.~A.~R\"osinger and C.~W.~Scherer, ``Lifting to passivity for $\mathcal{H}_2$-gain-scheduling synthesis with full block scalings,'' \emph{IFAC-PapersOnLine}, vol.~53, no.~2, pp. 7292--7298, 2020.

\bibitem{scherer2013}
C.~W.~Scherer, ``Structured {$H_\infty$}-optimal control for nested interconnections: A state-space solution,'' \emph{Syst. Control Lett.}, vol.~62, no.~12, pp. 1105--1113, 2013.

\bibitem{scherer2014}
C.~W.~Scherer, ``{$H_\infty$}- and {$H_2$}-synthesis for nested interconnections: A direct state-space approach by linear matrix inequalities,'' in \emph{21st Int. Symp. Math. Theory of Netw. and Syst.}, 2014, pp. 1589--1594.

\bibitem{lessard2015}
L.~Lessard and S.~Lall, ``Optimal control of two-player systems with output feedback,'' \emph{IEEE Trans. Autom. Control}, vol.~60, no.~8, pp. 2129--2144, 2015.

\bibitem{scherer2000}
C.~W.~Scherer, ``Robust mixed control and linear parameter-varying control with full block scalings,''
in \emph{Advances in linear matrix inequality methods in control}, L.~El~Ghaoui and S.-I.~Niculescu, Eds.\, Philadelphia: SIAM, 2000, pp. 187--207.

\bibitem{zhou1996}
K.~Zhou, J.~C.~Doyle and K.~Glover, \emph{Robust and optimal control}.\, Englewood Cliffs, New Jersey: Prentice Hall, 1996, ch. 10.

\bibitem{pagfer00}
F.~Paganini and E.~Feron, ``Linear matrix inequality methods for robust $H_2$ analysis: {A} survey with comparisons,'' in \emph{Advances in linear matrix inequality methods in control}, L.~El~Ghaoui and S.-I.~Niculescu, Eds.\, Philadelphia: SIAM, 2000, pp. 129--151.

\bibitem{masubuchi1998}
I.~Masubuchi, A.~Ohara, and N.~Suda, ``{LMI}-based controller synthesis: A unified formulation and solution,'' \emph{Int. J. Robust Nonlin.}, vol.~8, no.~8, pp. 669--686, 1998.

\bibitem{scherer1997}
C.~W.~Scherer, P.~Gahinet, and M.~Chilali, ``Multiobjective output-feedback control via {LMI} optimization,'' \emph{IEEE Trans. Autom. Control}, vol.~42, no.~7, pp. 896--911, 1997.

\bibitem{gohsaf1995}
K.-C.~Goh and M.~G.~Safonov, ``Robust analysis, sectors, and quadratic functionals,'' in \emph{Proc. 34th IEEE Conf. Decision and Control}, 1995, pp. 1988--1993.

\bibitem{iwasaka1998}
T.~Iwasaki and S.~Hara, ``Well-posedness of feedback systems: Insights into exact robustness analysis and approximate computations,'' \emph{IEEE Trans. Autom. Control}, vol.~43, no.~5, pp. 619--630, 1998.

\bibitem{roesinger2019a}
C.~A.~R\"osinger and C.~W.~Scherer, ``A flexible synthesis framework of structured controllers for networked systems,'' \emph{IEEE Trans. Control Netw. Syst.}, vol.~7, no.~1, pp. 6--18, 2020.

\bibitem{Sch06}
C.~W.~Scherer, ``LMI relaxations in robust control,'' \emph{Eur. J. Control}, vol.~12, no.~1, pp. 3--29, 2006.

\bibitem{VeeSch16a}
J.~Veenman, C.~W.~Scherer, and H.~K\"{o}ro\u{g}lu, ``Robust stability and performance analysis based on integral quadratic constraints,'' \emph{Eur. J. Control}, vol.~31, pp. 1--32, 2016.
	
\bibitem{lofberg2004}
J.~L\"ofberg, ``YALMIP : A toolbox for modeling and optimization in MATLAB,'' in \emph{Proc. CACSD Conf.}, 2004, pp. 284--289.


\end{thebibliography}
\end{document}